\documentclass[11pt]{amsart}
\usepackage{enumerate}
\usepackage{geometry}                
\usepackage{graphicx}
\usepackage{amssymb}
\usepackage{epstopdf}
\usepackage{geometry} 
\usepackage{graphicx}
\usepackage{amsmath} 
\usepackage{stackengine}
\usepackage{romanbar}
\numberwithin{equation}{section}
\usepackage{commath}
\usepackage{cleveref}
\usepackage[normalem]{ulem}
\usepackage{bbm}
\usepackage{subfig}
\usepackage{float}
\usepackage{url}
\usepackage{mathtools}
\usepackage{amsthm} 
\usepackage[disable]{todonotes} 
\DeclareMathOperator{\sgn}{sgn}

\newcommand{\R}{\mathbb R}
\newcommand{\V}{\R^{2}}

\theoremstyle{plain}
\newtheorem{theorem}{Theorem}[section]
\newtheorem*{theorem*}{Theorem}
\newtheorem{lemma}[theorem]{Lemma}
\newtheorem{corollary}[theorem]{Corollary}
\newtheorem{proposition}[theorem]{Proposition}
\newenvironment{claim}[1]{\par\noindent\underline{Claim.}\space#1}{}
\newenvironment{claimproof}[1]{\par\noindent\emph{Proof of Claim.}\space#1}{\leavevmode\unskip\penalty9999 \hbox{}\nobreak\hfill\quad\hbox{$\blacksquare$}}
\theoremstyle{remark}
\newtheorem*{remark}{Remark}
\theoremstyle{remark}
\newtheorem{question}[theorem]{Question}
\theoremstyle{definition}
\newtheorem{definition}[theorem]{Definition}
\allowdisplaybreaks
\title[Nonexistence of $T_4$]{Nonexistence of $T_4$ configurations for hyperbolic systems and the Liu entropy condition}
\author[Sam G. Krupa]{Sam G.  Krupa}
\address[Sam G. Krupa]{\newline Max Planck Institute for Mathematics in the Sciences, \newline 04103 Leipzig, Germany}
\email{Sam.Krupa@mis.mpg.de}

\author[L\'{a}szl\'{o} Sz\'{e}kelyhidi, Jr.]{L\'{a}szl\'{o} Sz\'{e}kelyhidi, Jr.}
\address[L\'{a}szl\'{o} Sz\'{e}kelyhidi, Jr.]{\newline Institut  f{\"u}r Mathematik\\ Universit{\"a}t Leipzig, \newline D-04109, Leipzig\\ Germany \newline \textbf{Also:}
\newline Max Planck Institute for Mathematics in the Sciences, \newline 04103 Leipzig, Germany}
\email{laszlo.szekelyhidi@math.uni-leipzig.de}
\email{Laszlo.Szekelyhidi@mis.mpg.de}

\thanks{Both authors worked on this paper during a stay at the Hausdorff Research Institute for Mathematics (HIM) at the University of Bonn during the Evolution of Interfaces Trimester Program (during Spring 2019). The stay was supported by the HIM. Both authors would like to thank the HIM for the support and the nice working atmosphere. 
This project has received funding from the European Research Council (ERC) under the European Union’s Horizon 2020 research and innovation programme (Grant Agreement No. 724298).
This material is based upon work supported by the National Science
Foundation under Grant No. DMS-1926686. This work was partially supported by NSF Grant DMS-1614918. The first author was also partially supported by NSF-DMS Grant 1840314.}

\date{\today}                                           

\begin{document}
\keywords{Conservation laws, one space dimension, strict hyperbolicity, Liu entropy condition, entropy conditions, entropy solutions, convex integration, differential inclusion, non-uniqueness.}
\subjclass[2020]{Primary 35L65; Secondary 35L40, 35D30}
\begin{abstract}
    We study the constitutive set $\mathcal{K}$ arising from a $2\times 2$ system of conservation laws in one space dimension, endowed with one entropy and entropy-flux pair. The convexity properties of the set $\mathcal{K}$ relate to the well-posedness of the underlying system and the ability to construct solutions via convex integration. Relating to the convexity of $\mathcal{K}$, in the particular case of the $p$-system,   Lorent and  Peng [{\em Calc. Var. Partial Differential Equations}, 59(5):Paper No. 156, 36, 2020] show that $\mathcal{K}$ does not contain $T_4$ configurations. Recently, Johansson and Tione [{\em arXiv
e-prints}, page arXiv:2208.10979, August 2022] showed that $\mathcal{K}$ does not contain $T_5$ configurations. 

In this paper, we provide a substantial generalization of these results, based on a careful analysis of the shock curves for a large class of $2\times 2$ systems. In particular, we provide several sets of hypothesis on general systems which can be used to rule out the existence of $T_4$ configurations in the constitutive set $\mathcal{K}$. In particular, our results show the nonexistence of $T_4$ configurations for every well-known $2\times 2$ hyperbolic system of conservation laws which verifies the Liu entropy condition.
\end{abstract}
\maketitle
\tableofcontents

\section{Introduction}
Consider the $2\times 2$ system of conservation laws in one space dimension,
\begin{align}\label{system}
\begin{cases}
\partial_t U+ \partial_x f(U)=0,\mbox{ for } x\in\mathbb{R},\mbox{ } t>0,\\
U(x,0)=U^0(x),
\end{cases}
\end{align}
for the unknown function $U\colon \mathbb{R}\times[0,\infty)\to\mathcal{V}$ with initial datum $U^0$. The open convex set $\mathcal{V}\subset \V$ is the \emph{state space} where the solution $U$ takes its values ($U$-space). The function $f:\mathcal{V}\to\V$ is called the \emph{flux}. In this paper, we often take $\mathcal{V}=\V$. 



For hyperbolic systems of conservation laws, weak solutions (in the sense of distributions) are the natural class in which to study the well-posedness. Indeed, even for the scalar conservation laws in one space dimension $C^\infty$ initial data can typically develop discontinuities in finite time. On the other hand it is well-known that weak solutions to \eqref{system} are highly non-unique, hence one needs selection criteria, e.g. based on entropy, to try to single out the physically relevant solution. This can be done in large generality for scalar problems, but for systems \emph{in the large} it is not known if entropy conditions can be used to obtain a well-defined selection principle. The present paper is a contribution to this general problem in the setting of $2\times 2$ systems with a single strictly convex entropy.

We recall that an entropy/entropy-flux pair for the system \eqref{system} is a pair of functions
$\eta\colon\mathcal{V}\to\R$ and $q\colon\mathcal{V}\to\mathbb{R}$ such that 
\begin{align}\label{entropy_compat_convention912}
    \nabla q=\nabla\eta Df.
\end{align}
Under this condition any \emph{smooth} solution $U$ of the system \eqref{system} automatically satisfies $\partial_t\eta(U)+\partial_xq(U)=0$. If $\eta$ is convex, a weak solution is said to be \emph{entropic} (for the entropy $\eta$) if $U$ verifies the entropy inequality
\begin{align}\label{entropy_ineq}
    \partial_t\eta(U)+\partial_x q(U)\leq 0
\end{align}
in the sense of distributions. \emph{The basic paradigm is that the class of entropic weak solutions enjoys better properties than the general class of weak solutions.} One prominent example is the seminal work of DiPerna on compactness of the class of entropy solutions for a large class of $2\times 2$ systems \cite{MR684413,MR719807}, following the pioneering work of Tartar on compensated compactness \cite{MR584398}. In particular DiPerna showed compactness (and consequently global existence of weak solutions) in large generality for small data as well as for large data under additional (global) assumptions (satisfied e.g.~for the $p$-system). Although the compactness of the class of entropy solutions does not imply uniqueness, there is a large body of subsequent work on differential inclusions and convex integration \cite{MR2008346,MR2600877,MR2564474,multi_d_illposed} to the effect that \emph{lack of compactness} in many situations leads to `wild' \emph{non-uniqueness}. Thus, \emph{the question of compactness of the class of entropy solutions is central to our understanding of \eqref{system}}. 

When considering entropic weak solutions, one possibility is to assume that \eqref{entropy_ineq} holds \emph{for all} entropy/entropy-flux pairs, i.e. all pairs $(\eta,q)$ which satisfy \eqref{entropy_compat_convention912}. This is a natural requirement, satisfied for instance by weak solutions arising in the inviscid limit $\partial_t U^{\varepsilon}+\partial_x (f(U^{\varepsilon}))=\varepsilon \Delta U^{\varepsilon}$ (indeed, the strong convergence $U^{\varepsilon}\to U$ is closely related to the compactness question, see \cite{MR584398}). In the works of DiPerna and Tartar, this point of view is taken and the existence of a large set of entropy/entropy-flux pairs is crucially used. On the other hand, for general $n\times n$ hyperbolic systems of conservation laws with $n\geq 3$ the existence of entropy/entropy-flux pairs is not clear at all as the system \eqref{entropy_compat_convention912} is over-determined. For many physical systems only one entropy exists (for integrability conditions on the entropy, see \cite[p.~13-14]{dafermos_big_book} and \cite[p.~54-55]{dafermos_big_book}). 

The compactness question for weak solutions of $2\times 2$ systems with a single convex entropy has been the subject of a number of works, most intensively for the $p$-system (in Lagrangian coordinates) or, equivalently, the system of isentropic gas-dynamics (in Eulerian coordinates) \cite{MR2008346,MR4144350,https://doi.org/10.48550/arxiv.2208.10979}. The key observation made in \cite{MR2008346} is that the system \eqref{system} with equality in \eqref{entropy_ineq} for a given entropy/entropy-flux pair $(\eta,q)$ can be equivalently reformulated as a first order differential inclusion of the type
$$
D\psi(x,t)\in \mathcal{K}_{f,\eta,q}\textrm{ a.e. } (x,t),
$$
where $\mathcal{K}_{f,\eta,q}$ is a given two dimensional surface embedded in the space of matrices $\R^{3\times 2}$ (see below in \Cref{solutions_by_convex_integration_section} and Equation \eqref{manifold_K}). Moreover, compactness properties of sequences of uniformly bounded approximate solutions of \eqref{system}-\eqref{entropy_ineq} are intimately connected to the rank-one geometry of $\mathcal{K}_{f,\eta,q}$. Such questions have also been the subject of intensive investigation in the context of elliptic systems \cite{MR1983780,MR2048569,MR4350101} -- and see again the survey \cite{MR2008346}. A main point in these works is to find or eliminate certain special 2-point configurations (rank-one connections) as well as special $N$-point configurations, $N\geq 4$ ($T_N$-configurations) -- see below in \Cref{sec:TNdef}: In a nutshell, the presence of such configurations implies lack of compactness, and in many cases lack of uniqueness/regularity, whereas the absence of such configurations is good indication (although not proof) of compactness. In the context of the $p$-system Lorent and Peng \cite{MR4144350} show that $\mathcal{K}_{f,\eta,q}$ does not contain $T_4$ configurations. This analysis was recently extended by Johansson and Tione \cite{https://doi.org/10.48550/arxiv.2208.10979}, who showed that $\mathcal{K}_{f,\eta,q}$ does not contain $T_5$ configurations. Both proofs are algebraic, and rely heavily on the specific form of $f,\eta,q$ given by the $p$-system. 

In this paper our purpose is to put the differential inclusions framework, introduced in \cite{MR2008346}, in the more general setting of $2\times 2$ systems with strictly hyperbolic and genuinely nonlinear flux $f$ and strictly convex entropy $\eta$. Our principal aim is to study the effect of various classically imposed (global) conditions on the flux or the entropy on the rank-one convex geometry of $\mathcal{K}_{f,\eta,q}$, specifically the absence of rank-one connections and $T_4$ configurations. Our main results, stated rigorously below in Section \ref{sec:hyp}, have the following immediate consequence:

\begin{theorem}\label{main_theorem2}
  Let $f,\eta,q$ be strictly hyperbolic and genuinely nonlinear, given by one of the following classical systems with their natural strictly convex entropy $\eta$:
  \begin{itemize}
  \item the $p$-system \eqref{psystem};
  \item isentropic Euler (with a general pressure law) \eqref{isentropicEulergeneral};
  \item the equations for an ideal gas or the system of shallow water waves \eqref{isentropicEuler};
  \item two coupled copies of Burgers \eqref{twoBurgers}.	
  \end{itemize}
 Then the set $\mathcal{K}_{f,\eta,q}$ \eqref{manifold_K} does not contain any $T_4$ configurations.
\end{theorem}





\subsection{Differential inclusions}\label{solutions_by_convex_integration_section}
 \hspace{.04in}
 
 For a given flux $f$, entropy $\eta$, and entropy-flux $q$, following \cite{MR2008346}, we consider stream functions $\psi(x,t)\colon\mathbb{R}^2\to\mathbb{R}^3$ such that 
 \begin{equation}
 \begin{aligned}\label{streaming}
 (u_1,-f_1(U))&=((\psi_1)_x,-(\psi_1)_t)\\
  (u_2,-f_2(U))&=((\psi_2)_x,-(\psi_2)_t)\\
 (\eta(U),-q(U))&=((\psi_3)_x,-(\psi_3)_t),
 \end{aligned}
 \end{equation}
 where we write  $U=(u_1,u_2)$, $\psi=(\psi_1,\psi_2,\psi_3)$ for the components of $\psi$, and we write $f$ in terms of its components $f=(f_1,f_2)$. 
In terms of $\psi$, the system
\begin{equation}
\begin{aligned}\label{e:system}
	\partial_tU+\partial_xf(U)&=0\\
	\partial_t\eta(U)+\partial_xq(U)&=0
\end{aligned}
\end{equation}
is equivalent to the first order differential inclusion
\begin{align}\label{dffinclusion}
D\psi\in \mathcal{K}_{f,\eta,q}, 
\end{align}
where the constitutive set $\mathcal{K}_{f,\eta,q}\subset \R^{3\times 2}$ is given by 
\begin{align}\label{manifold_K}
\mathcal{K}_{f,\eta,q}\coloneqq \Bigg\{
G(U) : (u_1,u_2)=U\in\mathcal{V}
\Bigg\},
\end{align}
and 
\begin{align}\label{matrixmap}
G(U)\coloneqq
\begin{bmatrix}
u_1 & f_1(U) \\
u_2 & f_2(U) \\
\eta(U) & q(U)
\end{bmatrix}.
\end{align}

\subsubsection{Rank-one connections}

For the general program of compactness for approximate solutions to differential inclusions of the type \eqref{dffinclusion} we refer to \cite{MR2008346,MR1731640}. Roughly speaking, compactness properties of the inclusion \eqref{dffinclusion} are closely linked to the rank-one geometry of the set $\mathcal{K}_{f,\eta,q}$. The simplest example arises if there exist $A,B\in\mathcal{K}_{f,\eta,q}$ such that $A-B$ is a rank-one matrix - such pairs are called \emph{rank-one connections}. In this case it is easy to see that the inclusion admits plane-wave solutions oscillating between $A$ and $B$, leading to a loss of compactness. In our context such configurations would correspond to jump discontinuities with zero entropy production (i.e. equality in \eqref{entropy_ineq}). 

In the elliptic context studied in \cite{MR1983780,MR2048569,MR4350101}, rank-one connections can be easily ruled out by looking at the tangent space of the associated set $K$: (linearized) ellipticity implies that the tangent cannot contain rank-one matrices. However, as noted in the survey \cite{MR2008346}, the set $\mathcal{K}_{f,\eta,q}$ is necessarily degenerate in the sense that the tangent space will always contain rank-one matrices (this follows from \eqref{entropy_compat_convention912}), and therefore ruling out rank-one connections ``in the large'' is more subtle.  
In the case of scalar equations or for the $2\times 2$ $p$-system, assuming strict hyperbolicity and genuinely nonlinearity (which are higher-order local conditions), it is not hard to check by direct calculation that such rank-one connections cannot occur in the associated $\mathcal{K}_{f,\eta,q}$. However, in the general case these conditions do not seem to suffice. Although we believe that this is well-known, we will dedicate \Cref{sec:norankone} to show that certain standard conditions in the literature, which apply to a large class of systems (e.g. the Liu entropy condition) are sufficient to rule out rank-one connections (see \Cref{norankone} in the next subsection).

\subsubsection{$T_N$ configurations}\label{sec:TNdef}

Even when no rank-one connections exist in $\mathcal{K}_{f,\eta,q}$, compactness might be lost. This fact has been observed in the context of compensated compactness by Tartar \cite{MR1320538}, in the context of elliptic regularity by Scheffer \cite{MR2624766}, and later crucially used in \cite{MR1983780,MR2048569}. The prime example is given by a $T_N$ configuration, defined as follows.
\begin{definition}[$T_N$ configuration]\label{T_N_def}
An ordered set of $N\geq 4$ matrices $\{X_i\}_{i=1}^N\subset\mathbb{R}^{m\times n}$ without rank-one connections to said to form a $T_N$ \emph{configuration} if there exist matrices $P$, $C_i\in\mathbb{R}^{m\times n}$ and real numbers $\kappa_i>1$ such that 
\begin{equation}
\begin{aligned}\label{T_N_param}
X_1&=P+\kappa_1C_1\\
X_2&=P+C_1+\kappa_2C_2\\
\vdots\\
X_N&=P+C_1+\cdots+C_{N-1}+\kappa_N C_N,
\end{aligned}
\end{equation}
and moreover $\rm{rank}(C_i)=1$ and $\sum_{i=1}^N C_i=0$.
\end{definition}

The relevance of this definition is the following fact, used in a number of papers including \cite{MR1983780,MR2048569,MR2008346,MR2118899}: if a set $K$ contains $N$ elements $X_1,\dots,X_N$ which form a $T_N$ configuration, then the associated differential inclusion exhibits loss of compactness of approximate solutions. Moreover, under some additional non-degeneracy conditions, in this case the inclusion will admit (non-unique) Lipschitz solutions which are nowhere $C^1$ - with the identification \eqref{streaming} this would correspond to bounded measurable weak solutions of \eqref{e:system} which are extremely rough in the sense that for a fixed time $t$, the solution with this fixed time $U(\cdot,t)$ at any point $x$ doesn't have strong one-sided traces (in space)  in $L^1$ (c.f. \cite{MR4487515}). 

In light of this, an important question, besides eliminating rank-one connections, is to eliminate $T_N$ configurations in $\mathcal{K}_{f,\eta,q}$, as has been done for the $p$-system for $N=4$ in \cite{MR4144350} and for $N=5$ in \cite{https://doi.org/10.48550/arxiv.2208.10979}, and in the elliptic context in
\cite{MR4350101}. The proof in \cite{https://doi.org/10.48550/arxiv.2208.10979} hinges on the fact that the $p$-system can be reformulated such that the matrix \eqref{matrixmap} (when restricted to the first two rows) is symmetric. Thus, the rank-one matrices $C_i$ in the definition of the $T_N$ configuration (see \eqref{T_N_param}) will be symmetric (again, when restricted to the first two rows). Further, due to being rank-one, each of the $C_i$ (when restricted to the first two rows) will have at least one zero eigenvalue. Thus each $C_i$ has a sign. The proof then proceeds by casework, considering each of the $2^N$ possible cases ($N=4,5$) for the signs of the $C_i$ matrices in a $T_N$ configuration. However, this proof does seem to be sensitive to the specific structure of the $p$-system, and seems difficult to generalize to perturbations, e.g.~\eqref{generalpsystem} or \eqref{twoBurgers}, because in both cases the symmetric structure is lost.

On the other hand, the techniques of Lorent-Peng \cite{MR4144350}, whilst also algebraic, have some hidden geometric structure, which we are able to utilize in our more general setting. In \Cref{sec:lorentpengcomp} we will discuss and compare similarities and differences between our methods and those of \cite{MR4144350}.

For more information on $T_N$ configurations we refer to \cite{MR2118899,MR2048569}. The only property of $T_N$ configurations that we will need is the following well-known proposition:
\begin{proposition}\label{TN_change_sign}
Let $\{X_i\}_{i=1}^N\subset \mathbb{R}^{m\times n}$ be a $T_N$ configuration. For every $2\times 2$ subdeterminant $M:\R^{m\times n}\to \R$ and for every $i$ the set $\{M(X_i-X_j)\colon j\neq i\}$ necessarily changes sign.
\end{proposition}
In our setting the set $\mathcal{K}_{f,\eta,q}$ consists of $3\times 2$ matrices. It will be convenient to introduce the following notation for the relevant subdeterminants: for a matrix $A\in \R^{3\times 2}$ denote by $\det_{rs}A$ the $2\times 2$ subdeterminant consisting of the rows $r,s\in \{1,2,3\}$ of $A$.


\subsection{Hyperbolicity, entropies, shock waves}\label{sec:shocks}

Let us now introduce some basic concepts from the theory of hyperbolic conservation laws. First we discuss standard conditions in the theory of systems of hyperbolic conservation laws, which can be stated as local (differential) constraints. Then, we will need to consider global conditions as well. 

\subsubsection{Admissible systems}

Following \cite{K-K} for our notation, the system  \eqref{system} is called \emph{strictly hyperbolic} if for every $U\in\mathcal{V}$ ($U=(u_1,u_2)$) the matrix $Df$ has two distinct real eigenvalues $\lambda_1(U)$ and $\lambda_2(U)$. We assume $\lambda_1<\lambda_2$. The eigenvalues are called \emph{characteristic speeds}. The corresponding eigenvectors $r_1$ and $r_2$ are then linearly independent. The eigenvectors are vector fields, and the integral curves $R^1$ and $R^1$ of these respective direction fields are called \emph{rarefaction curves}. We denote the curve $R^k$ which passes through a point $U_0$ in $U$-space by $R^k_{U_0}$.

The system \eqref{system} is said to be \emph{genuinely nonlinear} if each characteristic speed $\lambda_i$ varies strictly monotonically along the rarefaction curves $R^i$ of the same family, i.e. $r_i\cdot\nabla \lambda_i\neq0$. We choose the following normalizations:
\begin{align}
\abs{r_i}=1,\hspace{.4in} r_i\cdot\nabla\lambda_i>0, \hspace{.4in} i=1,2,
\end{align}
and similarly for the left eigenvectors $l_1,l_2$ of $Df$,
\begin{align}
\abs{l_i}=1,\hspace{.4in} l_i r_i>0, \hspace{.4in} i=1,2.
\end{align}
Remark that $l_ir_j=0$ for $i\neq j$.

We will sometimes use the \emph{Smoller-Johnson condition}, which can be expressed in terms of the eigenvectors:
\begin{align}
   l_j D^2 f(r_i ,r_i)>0, \hspace{.4in} j\neq i, \label{SJcond}
\end{align}
where $D^2 f(r_i ,r_i)=(r_i\cdot \nabla Df)r_i$ denotes the second Fréchet derivative of $f$ in the direction of $r_i$.

Condition \eqref{SJcond} was introduced by Smoller and Johnson \cite{MR236527}. They showed that it is in fact equivalent to the Glimm-Lax shock interaction condition \cite{MR0265767}. Geometrically, \eqref{SJcond} says that all rarefaction curves of both families are strictly convex. Indeed, if $\abs{\kappa_i}$ is the curvature of $R^i$, then
\begin{align}\label{kappa_formula}
    \kappa_i=\frac{l_j D^2 f(r_i,r_i)}{\lambda_i - \lambda_j},
\end{align}
and $\kappa_1<0,\kappa_2>0$ due to strict hyperbolicity and \eqref{SJcond}\footnote{To study the compactness of exact solutions in $BV_{\text{loc}}$ with large oscillation, DiPerna requires \cite[p.~38]{MR684413} that the system \eqref{system} admit a coordinate system of quasi-convex Riemann invariants in order to guarantee the existence of a strictly convex entropy in the large. We remark this is related to the Smoller-Johnson condition \eqref{SJcond}: A Riemann invariant is quasi-convex if and only if the system has convex rarefaction curves. On the other hand, \eqref{SJcond} ensures that the rarefaction curves are \emph{strictly} convex.}. Thus, we see that the $R^1$ curves bend toward $-l_2$ (away from $r_2$), while the $R^2$ curves bend toward $l_1$ (and $r_1$) (see \cite[p.~448]{K-K} for details).

We then have the definition for a large class of systems we will consider in this paper,

\begin{definition}[Admissible systems]\label{admsystem}
A $2\times2$ system of conservation laws which is strictly hyperbolic, genuinely nonlinear, and verifies the Smoller-Johnson condition \eqref{SJcond} will be called \emph{admissible}.
\end{definition}

Remark that we do not have to look far for a system which does not verify the Smoller-Johnson condition: two coupled copies of Burgers (see \eqref{twoBurgers}) have straight rarefaction curves and thus do not verify Smoller-Johnson. We will therefore introduce a different set of hypotheses (Hypotheses $(\mathcal{H}2)$ below), which does not need the Smoller-Johnson condition.

\subsubsection{Shock solutions and the Hugoniot locus}\label{sec:Hugoniotlocus}

For left- and right-hand states $U_L,U_R\in\mathcal{V}$, and $\sigma\in\mathbb{R}$, the function
\begin{align}
U(x,t)\coloneqq
  \begin{cases}
    U_L       & \quad \text{if } x<\sigma t,\\
    U_R  & \quad \text{if } x>\sigma t,
  \end{cases}\label{shocksol}
\end{align}
is a weak solution to the system \eqref{system} if and only if the triple $(U_L,U_R,\sigma)$ verify the Rankine-Hugoniot jump condition,
\begin{align}\label{RHcond}
    \sigma(u_L-u_R)=f(u_L)-f(u_R),
\end{align}
in which case we call \eqref{shocksol} a \emph{shock solution} (and $(U_L,U_R)$ is a \emph{shock}) with \emph{shock speed} $\sigma$. The solution \eqref{shocksol} is entropic for the entropy $\eta$ if and only if
\begin{align}\label{entropicshock}
    q(U_R)-q(U_L)\leq \sigma (\eta(U_R)-\eta(U_L)).
\end{align}


At the level of the set $\mathcal{K}_{f,\eta,q}$ (see \eqref{manifold_K}), $(U_L,U_R)$ verifying the Rankine-Hugoniot condition can be equivalently stated as
\begin{equation*}
  \textrm{det}_{12}(G(U_L)-G(U_R))=0,
\end{equation*}
where recall \eqref{matrixmap} and that we denote by $\det_{12}$ the subdeterminant consisting of the first two rows. On the other hand observe that if the shock verifies \eqref{entropicshock} with strict inequality, then $G(U_L)-G(U_R)$ is not rank-one.

From the Rankine-Hugoniot jump condition, at a point $U_0$ we can define the \emph{Hugoniot locus} as
\begin{align}\label{locus}
    H(U_0)\coloneqq \{\hspace{.04in}U \hspace{.04in}|\hspace{.04in} \exists \sigma : \sigma(U-U_0)=f(U)-f(U_0)\}.
\end{align}
For any $2\times 2$ strictly hyperbolic system of conservation laws, locally around $U_0$, the set $H(U_0)$ consists of the union of two smooth curves $S^1_{U_0}$ and $S^2_{U_0}$,  each passing through the point $U_0$ (for details, see for example \cite[Theorem 2, p.~583]{MR1625845}). These are the \emph{shock curves} of the Hugoniot locus. For many physical systems, the shock curves will exist globally. We will only consider such systems. In this paper, we will always smoothly parameterize $S^1_{U_0}$ and $S^2_{U_0}$ as follows: $S^1_{U_0}=S^1_{U_0}(s)$ and $S^2_{U_0}=S^2_{U_0}(s)$ with $S^1_{U_0}(0)=S^2_{U_0}(0)=U_0$. We also choose a smooth parameterization for the $\sigma$, i.e. $\sigma^1_{U_0}$ and $\sigma^2_{U_0}$ such that
\begin{align}
    \sigma^k_{U_0}(s)(S^k_{U_0}(s)-U_0)=f(S^k_{U_0}(s))-f(U_0),
\end{align}
for $k=1,2$. 

\hspace{.1in}

In \Cref{sec:hyp} below, we will introduce two sets of hypotheses on the system \eqref{system}, in order to cover a large range of physical systems. 
 For such systems we will be able to conclude that \eqref{entropicshock} always holds with strict inequality. Consequently: 
\begin{proposition}[No rank-one connections in $\mathcal{K}_{f,\eta,q}$]\label{norankone}
Consider any system \eqref{system} with global shock curves, verifying the global Liu entropy condition, and endowed with a strictly convex entropy $\eta$ with associated  entropy-flux $q$. Then the set $\mathcal{K}_{f,\eta,q}$ does not contain rank-one connections.
\end{proposition}

\begin{remark}
For the \emph{Liu entropy condition}, see \cite[Section 8.4]{dafermos_big_book}. We also introduce it below (see Hypotheses $(\mathcal{H}2)$ \eqref{item1}). Roughly speaking, it says that the shock speed is varying monotonically along the shock curves in the Hugoniot locus.  For a broad class of conservation laws, it is possible to show the equivalence of genuine nonlinearity and the strict and global Liu entropy condition (as suggested by Liu \cite{MR369939}).
\end{remark}

The proof of \Cref{norankone} is in \Cref{sec:norankone}. Remark \Cref{norankone} generalizes one direction of the result \cite[Proposition 4]{MR4144350}, which applies only to the $p$-system. A key idea in this paper is to use global knowledge of the geometry of shock curves and the Hugoniot locus $H(U_0)$ to divide state space into four connected components, where in each component $\det_{12}(G(U)-G(U_0))$ has constant sign (recall \eqref{matrixmap}), and to combine this with \Cref{TN_change_sign} to derive a contradiction.

\subsection{Genuine nonlinearity and the Liu entropy condition}
\hspace{.04in}

One question we were not able to answer in this paper is
\begin{question}\label{questionc}
Does there exist a strictly hyperbolic $2\times 2$ system \eqref{system} with flux $f$, endowed with a strictly convex entropy $\eta$, and associated entropy-flux $q$, such that the Liu entropy condition holds along all shock curves, and the  constitutive set for this system (see \eqref{manifold_K}) admits a $T_4$ configuration?
\end{question}

While the present paper is able to answer in the negative for a large variety of systems, a proof of a negative result for a general system remains elusive. We conjecture the answer is \emph{no}.

\Cref{questionc} is of considerable interest. If we could construct an $f$, $\eta$, and $q$ that would allow for solutions via convex integration-type techniques, as discussed above, such solutions would not have any sort of strong traces, and in particular would exhibit  oscillation of order one in every open subset of space-time (see e.g. \cite{MR2048569}).

For scalar conservation laws in one space dimension, genuine nonlinearity, coupled with just a \emph{single strictly convex entropy}, causes solutions with arbitrary $L^\infty$ initial data to be instantaneously regularized to $BV_{\text{loc}}$, which in particular means the solutions have, for each fixed time, left and right limits in space (see \cite{panov_uniquness,delellis_uniquneness,2017arXiv170905610K}).


On the other hand, in the case of \emph{systems} of conservation laws in one space dimension, the situation is not so clear cut. In an active research program Vasseur and collaborators aim to show selection criteria to get uniqueness of solutions. The program considers systems having properties related to the Liu entropy condition, with a strictly convex entropy. The solutions under consideration are not assumed to be small in $BV$, and may in fact be large $L^2$ perturbations. However, they are assumed to have \emph{strong traces}, a property that is certainly not satisfied by solutions constructed via convex integration. For a survey, see \cite{MR3475284} -- and see also \cite{MR4487515,Leger2011}. 

Of particular note, our work in the present paper applies to the equation of isentropic flow of an ideal gas (see \eqref{isentropicEuler}) with $\gamma=3$ -- a genuinely nonlinear system. This particular system has been the focus of intense study. In \cite{MR1720782}, Vasseur is able to show that solutions to the system are regular in time. Moreover, interesting recent work by Golding \cite{golding} proves that the system exhibits additional regularizing effects on solutions. In \cite{golding}, it is shown that this particular system with $\gamma=3$ has a regularizing effect on solutions similar to the genuinely nonlinear, multidimensional scalar case.  The result gives the first example of a non-degenerate system, i.e. not Temple class, which has the property that $L^\infty$ initial data is regularized, leading to entropic solutions with substantial regularity. It should be noted that the results of Golding \cite{golding} do not quite manage to show that solutions automatically regularize enough to have the strong traces. Moreover, the works \cite{MR1720782,golding} utilize the kinetic formulation, and they require an infinite family of entropies to derive their results. 

To conclude this digression, a positive answer to \Cref{questionc} would suggest that the works \cite{MR1720782,golding} cannot extend to the systems case with a single entropy.  It would also suggest the hypothesis that solutions have strong traces is truly necessary and not merely a technical condition. 

\subsection{Plan for the paper}

The plan for the rest of the paper is the following. In \Cref{sec:hyp}, we introduce the two sets of hypotheses on the system \eqref{system} which we will work with in this paper (Hypotheses $(\mathcal{H}1)$ and $(\mathcal{H}2)$) and state our main results (\Cref{main_theorem1}, \Cref{main_theorem}, and \Cref{changeofCoordsnoT4}). In \Cref{sec:systemsverify}, we will check that various systems verify the conditions of one of these sets of hypotheses or the other. Finally, in \Cref{sec:proofshyp}, we present the proofs that the Hypotheses $(\mathcal{H}1)$ or $(\mathcal{H}2)$ imply nonexistence of $T_4$.




\section{Main Theorems: Nonexistence of $T_4$ under Hypotheses $(\mathcal{H}1)$ or $(\mathcal{H}2)$}\label{sec:hyp}

In this section we will introduce two sets of Hypotheses ($(\mathcal{H}1)$ and $(\mathcal{H}2)$) which we will consider for the system \eqref{system}. We introduce these particular hypotheses for three reasons: (1) they use conditions which are used in the literature, see for example \cite{K-K,dafermos_big_book,serre_book}, (2) physical systems will fall under one set of the hypotheses or the other (see \Cref{sec:systemsverify}), and (3) under either Hypotheses $(\mathcal{H}1)$ or $(\mathcal{H}2)$, the constitutive set $\mathcal{K}_{f,\eta,q}$ will not contain any rank-one connections (see \Cref{norankone}, introduced earlier).

\subsection{Sector condition}

The class of admissible systems (\Cref{admsystem}) is a very general class. Systems arising from physics have additional properties which we will need to utilize to deal with the differential inclusion \eqref{dffinclusion}.

\begin{definition}[\protect{Sector condition \cite[p.~472]{K-K}}]\label{sector_def}
An admissible system verifies the \emph{sector condition} if there exists two fixed linearly independent vectors $w_1$ and $w_2$ such that for all $U$ the following inequalities hold: 
\begin{align*}
r_1(U)\cdot w_1 <0,& \hspace{.5in} r_1(U)\cdot w_2>0 \\
r_2(U)\cdot w_i >0,& \hspace{.5in} i=1,2. 
\end{align*}
\end{definition}

The sector condition is a global condition. In practice, for many systems it may be easier to check a local condition called the \emph{opposite variation condition} \cite[p.~450]{K-K}.  A large class of conservation laws enjoy opposite variation, including all nonlinear wave equations. For admissible systems in which rarefaction curves of opposite families always intersect \cite[p.~451]{K-K}, opposite variation implies the sector condition -- see \cite[Theorem 7.5]{K-K}. As noted by \cite[p.~451]{K-K}, most conservation laws possess this intersection property. However, Smoller \cite{MR247283} has produced examples of systems which do not.


\hspace{.1in}

Let us now consider an admissible system verifying the sector condition (\Cref{sector_def}) and fix a strictly convex set $K$ with smooth boundary. We can partition the boundary of $K$ into four connected pieces: $\Romanbar{I}$, $\Romanbar{II}$, $\Romanbar{III}$ and $\Romanbar{IV}$, with the properties that
\begin{align}
    \begin{cases}\label{decomp_K_def}
        \bullet \text{ } \partial K=\Romanbar{I}\cup \Romanbar{II} \cup \Romanbar{III} \cup \Romanbar{IV},\\
        \bullet \text{ For any $U\in\Romanbar{I}$, $r_2(U)$ points inward into $K$,}\\
        \bullet \text{ For any $U\in\Romanbar{II}$, $r_1(U)$ points inward into $K$,}\\
        \bullet \text{ For any $U\in\Romanbar{III}$, $-r_2(U)$ points inward into $K$,}\\
        \bullet \text{ For any $U\in\Romanbar{IV}$, $-r_1(U)$ points inward into $K$.}
    \end{cases}
\end{align}
If $K$ is noncompact, then $\Romanbar{I}$, $\Romanbar{II}$, $\Romanbar{III}$ and/or $\Romanbar{IV}$ may be the empty set.

Such a partition is not unique, but will always exist. Indeed, one possible choice is given by
\begin{equation}
\begin{aligned}\label{canonical_decomp}
    \Romanbar{I}^{\mathrm{o}} \coloneqq & \Big\{U\in\partial K \hspace{.04in}|\hspace{.04in} \nu(U) \in W^{--}\Big\},\\
    \Romanbar{II}^{\mathrm{o}} \coloneqq & \Big\{U\in\partial K \hspace{.04in}|\hspace{.04in} \nu(U) \in W^{+-}\Big\},\\
    \Romanbar{III}^{\mathrm{o}} \coloneqq & \Big\{U\in\partial K \hspace{.04in}|\hspace{.04in} \nu(U) \in W^{++}\Big\},\\    
    \Romanbar{IV}^{\mathrm{o}} \coloneqq & \Big\{U\in\partial K \hspace{.04in}|\hspace{.04in} \nu(U) \in W^{-+}\Big\},
\end{aligned}
\end{equation}
where $\nu(U)$ is the outward-pointing unit normal vector to the boundary of $K$ at the point $U$ and 
\begin{align*}
    W^{++}\coloneqq & \Big\{ \beta_1 w_1+\beta_2 w_2 \,|\, \beta_1>0, \beta_2>0\Big\},\\
    W^{+-}\coloneqq & \Big\{ \beta_1 w_1+\beta_2 w_2 \,|\, \beta_1>0, \beta_2<0\Big\},\\
    W^{-+}\coloneqq & \Bigg\{ \beta_1 w_1+\beta_2 w_2 \,|\, \beta_1<0, \beta_2>0\Big\},\\
    W^{--}\coloneqq & \Bigg\{ \beta_1 w_1+\beta_2 w_2 \,|\, \beta_1<0, \beta_2<0\Big\}.
\end{align*}
Here $w_1, w_2$ are the fixed vectors given by the sector condition in Definition \ref{sector_def}. Observe that $\Romanbar{I}^{\mathrm{o}}$,\dots,$\Romanbar{IV}^{\mathrm{o}}$ are by definition relatively open connected subsets of $\partial K$; to define $\Romanbar{I}$,\dots$\Romanbar{IV}$ one merely needs to assign the boundary points in a cyclic manner.

The choice \eqref{canonical_decomp} is presented to show that a decomposition with properties \eqref{decomp_K_def} always exists. However, in practice such decomposition is not unique (indeed, the choice of $w_1,w_2$ may not be unique), and we will make use of this flexibility to help ensure that our hypotheses (e.g.~Hypotheses $(\mathcal{H}1)$ \eqref{item51}) will hold.

\subsection{Hypotheses $(\mathcal{H}1)$ and $(\mathcal{H}2)$}

Consider an entropy $\eta$ and entropy-flux $q$ for the system \eqref{system}. For $c=(c_1,c_2)\in\mathbb{R}^2$, the following functions will play an important role in Hypotheses $(\mathcal{H}1)$ and $(\mathcal{H}2)$:
\begin{equation}
\begin{aligned}\label{tdef}
\tilde{\eta}(U)&\coloneqq \eta(U)+c\cdot U,\\
\tilde{q}(U)&\coloneqq q(U)+c\cdot f(U).
\end{aligned}
\end{equation}
Observe that $\tilde \eta, \tilde q$ are also an entropy, entropy-flux pair for our system, and any weak solution of \eqref{system}-\eqref{entropy_ineq} automatically satisfies \eqref{entropy_ineq} also with $\tilde\eta,\tilde q$.

We then consider the following hypotheses on the system \eqref{system}.

\textbf{Hypotheses $(\mathcal{H}1)$ on the system}
\begin{enumerate}[(i)]
\item The system is admissible (\Cref{admsystem}). \label{item11}
\item The system verifies the sector condition (\Cref{sector_def}) with two vectors $w_1,w_2\in\mathbb{R}^2$.\label{item21}
\item The characteristic speeds verify $\lambda_1(U) \leq 0 \leq \lambda _2(U)$ for all $U$.\label{item61}
\item The system is endowed with a strictly convex entropy $\eta$. \label{item41}
\item For all $c_1,c_2\in\mathbb{R}$, each level set $\{\tilde\eta=C\}$ (for $C\in\mathbb{R}$ which is not a global minimum of $\tilde\eta$) can be decomposed into four pieces as in \eqref{decomp_K_def} with the property that among the four sets $\tilde q (\Romanbar{I})$, $\tilde q (\Romanbar{II})$, $\tilde q (\Romanbar{III})$, and $\tilde q (\Romanbar{IV})$, the only nontrivial intersections are between the pairs $\tilde q (\Romanbar{I}),\tilde q (\Romanbar{III})$ and $\tilde q (\Romanbar{II}),\tilde q (\Romanbar{IV})$. Remark that for a set $K\subset\mathcal{V}$, we define $\tilde q(K)\coloneqq \{\hspace{.04in} \tilde q(U) \hspace{.04in}| \hspace{.04in}U\in K\}$. \label{item51}
\end{enumerate}

\begin{remark}

\hspace{.1in}

\begin{itemize}
    \item The level sets of $\tilde\eta$ are convex sets (which are potentially unbounded). Furthermore, since $\tilde\eta$ is strictly convex, it has at most one critical point and hence for any $C>\inf_{\R^2}\tilde\eta$ the level set $\{\tilde\eta=C\}$ is a one-dimensional manifold which is the boundary of the convex set $\{\tilde\eta<C\}$. 

    
    
    \item Item \eqref{item51} is a natural assumption on the entropy, entropy-flux pair. Indeed, assume that the global minimum of $\tilde\eta$ exists and consider a close-by level set of $\tilde\eta$. Since $\nabla^2\tilde q = \nabla^2 \tilde\eta Df +\nabla\tilde\eta D^2f$, at the global minimum we have $\nabla\tilde\eta=0$, $\nabla \tilde q=0$ and $\nabla^2\tilde q=\nabla^2 \tilde\eta Df$. Thus, for any general system with a strictly convex entropy and strictly signed characteristic speeds (one positive, one negative -- compare with Hypotheses $(\mathcal{H}1)$ \eqref{item61}), $\det \nabla^2\tilde q<0$ at this point and we have a saddle point, with level sets locally looking like hyperbolas. Remark that at the saddle point,
    \begin{align}
    r_i^\top \nabla^2\tilde q r_i = \lambda_i r_i^\top\nabla^2 \tilde\eta r_i,
\end{align}
for $i=1,2$. Thus, recalling that $\nabla \tilde q=0$ and that $\eta$ is strictly convex (which makes $\nabla^2 \tilde\eta$ positive-definite) we see that $\tilde q$ is concave in the direction of $r_1$ and convex in the direction of $r_2$. Thus, we can zoom in at this fixed saddle point, and we can choose vectors $w_1$ and $w_2$ (as in \Cref{sector_def}) such that $r_1$ and $r_2$, at this fixed point, meet the sector condition \emph{and} we have a decomposition of the local level sets of $\tilde\eta$ as in Hypotheses $(\mathcal{H}1)$ \eqref{item51}. We conclude that locally we will always have Hypotheses $(\mathcal{H}1)$ \eqref{item51}. By requiring Hypotheses $(\mathcal{H}1)$ \eqref{item51}, we are simply asking that this condition hold not only locally, but also globally for our entropy, entropy-flux pair.
\end{itemize}
\end{remark}

Our first main result is the following:
\begin{theorem}[Nonexistence of $T_4$ for systems verifying Hypotheses $(\mathcal{H}1)$] \label{main_theorem1}
For any conservation law \eqref{system} with entropy $\eta$ and entropy-flux $q$ and verifying Hypotheses $(\mathcal{H}1)$, the corresponding constitutive set \eqref{manifold_K} does not contain a $T_4$ configuration (see \Cref{T_N_def}).
\end{theorem}
\begin{remark}
The result also holds when the system \eqref{system} verifies the Smoller-Johnson condition \eqref{SJcond} but with the sign flipped, i.e. $l_j D^2 f(r_i ,r_i)<0$.
\end{remark}

We postpone the proof of \Cref{main_theorem1} to \Cref{sec:proofmain_theorem1}.

\vspace{.1in}

For systems not falling in the framework of Hypotheses $(\mathcal{H}1)$, we start with a very weak condition on the Hugoniot locus, which aims to rule out some strange behavior of shock curves which is not exhibited by any physical systems we are aware of. See \Cref{perversefig} for an example of the type of behavior we are trying to avoid.

\begin{figure}[tb]
      \includegraphics[width=.5\textwidth]{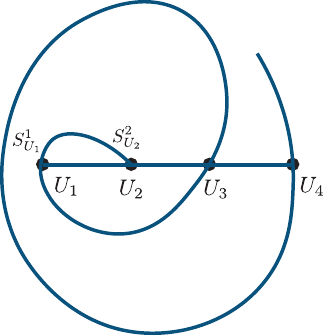}
  \caption{An example of the type of perverse behavior of shock curves which the non-perverse Hugoniot locus condition (\Cref{nonperverse}) rules out.}\label{perversefig}
\end{figure}

\begin{definition}[Non-perverse Hugoniot locus]\label{nonperverse}
We say the system of conservation laws \eqref{system} verifies the \emph{non-perverse Hugoniot locus} condition if the following holds:

If four points $U_1,\ldots,U_4$ in $U$-space are on a straight line, and each point is connected to the other three by a shock (see \eqref{RHcond}), then there is a fixed $k\in\{1,2\}$ such that the image of $S^k_{U_i}$ contains $U_j$, for all $i\neq j$. Furthermore, along each of the curves $S^k_{U_i}$, the Liu entropy condition holds, i.e. $\frac{\text{d}}{\text{ds}}\sigma^k_{U_i}(s)\neq 0$ for $s\neq 0$, and the Lax E-condition holds, i.e.
\begin{align}\label{laxeHyP}
    \sigma^k_{U_i}(s)\in I[\lambda_k(S^k_{U_i}(s)),\lambda_k(U_i)],
\end{align}
for all $s$ and where $I[a,b]$ denotes the closed interval with endpoints $a$ and $b$. Lastly, we require that for each of the shock curves $S^k_{U_i}$, at least one of the first or second coordinates (in $U$-space) of the shock curve are monotonic functions of $s$. 
\end{definition}

We require far less knowledge about the ``hyperbolic part'' of the system \eqref{system} (i.e., knowledge of the flux $f$: geometric knowledge of the shock curves, knowledge of eigenvalues, et cetera) if we can impose a stricter condition on the entropy $\eta$ and the entropy-flux $q$. In particular, we consider the following hypotheses.

\textbf{Hypotheses $(\mathcal{H}2)$ on the system}
\begin{enumerate}[(i)]
\item (global shock curves) For a fixed $U_0$, the Hugoniot locus $H(U_0)$ (see \eqref{locus}) is the union of two smooth curves $S^1_{u_L}$ and $S^2_{u_L}$, whose union cuts $\mathcal{V}$ into four connected pieces. Furthermore, the curves admit smooth parameterizations defined on the real line: $S^1_{U_0}=S^1_{U_0}(s)$ and $S^1_{U_0}=S^2_{U_0}(s)$ with $S^1_{U_0}(0)=S^2_{U_0}(0)=U_0$ and such that the map $(U_0,s)\mapsto S^k_{U_0}(s)$ is continuous (for $k=1,2$). There is also a smooth parameterization for the shock speed, i.e. $\sigma^1_{U_0}$ and $\sigma^2_{U_0}$ such that
\begin{align}
    \sigma^k_{U_0}(s)(S^k_{U_0}(s)-U_0)=f(S^k_{U_0}(s))-f(U_0),
\end{align}
for $k=1,2$. \label{item2},
\item (the Liu entropy condition) $\frac{\text{d}}{\text{ds}}\sigma^k_{u_L}(s)\neq 0$ for $k=1,2$ and for all $u_L$. \label{item1}
\item The non-perverse Hugoniot locus condition (\Cref{nonperverse}) is verified. \label{item5}
\item The system is endowed with a strictly convex entropy $\eta$. \label{item4}
\item For all $c_1,c_2\in\mathbb{R}$, the function $\tilde q$ restricted to a level set of $\tilde\eta$ (see \eqref{tdef}) has at most four critical points which are (local) extrema if the level set of $\tilde\eta$ is bounded, and less than four otherwise.  \label{item3}
\end{enumerate}

\begin{remark}
\hspace{.1in}

\begin{itemize}
    \item  Under Hypotheses $(\mathcal{H}2)$, we do not require the system be admissible, and we do not require the sector condition. In particular, we do not require genuine nonlinearity or the Smoller-Johnson condition. Instead, we only require the Liu entropy condition. Remark that we require Hypotheses $(\mathcal{H}2)$ \eqref{item2} because we are requiring so little on the system, not even strict hyperbolicity, that knowledge of the shock curves must be assumed. 
    \item For an example of some of the behavior of shock curves which is allowed under Hypotheses $(\mathcal{H}2)$, see \Cref{four_crit_proof_fig}.
    \item For more on the \emph{Liu entropy condition}, see \cite[Section 8.4]{dafermos_big_book}. 
    \item Due to the Liu entropy condition (Hypotheses $(\mathcal{H}2)$ (\ref{item1})), for each $k$ and $u$ the shock curves $S^k_u$ cannot self-intersect.     
\end{itemize}
\end{remark}

Our second main result is the following:
\begin{theorem}[Nonexistence of $T_4$ for systems verifying Hypotheses $(\mathcal{H}2)$] \label{main_theorem}
For a $2\times2$ system of conservation laws with flux $f$, entropy $\eta$ and entropy-flux $q$ verifying Hypotheses $(\mathcal{H}2)$, the set $\mathcal{K}_{f,\eta,q}$ cannot contain a $T_4$ configuration (see \Cref{T_N_def}).
\end{theorem}

We remark that the requirement that the Liu entropy condition (Hypotheses $(\mathcal{H}2)$ \eqref{item1}) hold can be weakened somewhat. The Liu entropy condition implies the nonexistence of rank-one connections in the set $\mathcal{K}_{f,\eta,q}$ (see \Cref{norankone}). However, for the proof of \Cref{main_theorem}, it is enough to simply assume this consequence, which is, as explained earlier in Section \ref{sec:Hugoniotlocus}, equivalent to a strict inequality in \eqref{entropicshock} for every shock $(U_L,U_R)$ with $U_L\neq U_R$. Indeed, such a condition is also considered by DiPerna in \cite[p.~39]{MR684413}, where he notes that this is ``typically the situation in mechanics.'' However, to simplify the statement of our results, we choose to include the Liu entropy condition in our Hypotheses $(\mathcal{H}2)$ as it is easily verified in all examples of interest.

The proof of \Cref{main_theorem} is again deferred to \Cref{sec:proofmain_theorem}.

\subsection{The hyperbolic part of Hypotheses $(\mathcal{H}1)$ implies the hyperbolic part of  $(\mathcal{H}2)$}

In this section, we show that the ``hyperbolic part'' of Hypotheses $(\mathcal{H}1)$ (parts \eqref{item11},\eqref{item21}, \eqref{item61})  implies the ``hyperbolic part'' of Hypotheses $(\mathcal{H}2)$ (parts \eqref{item2},\eqref{item1}, \eqref{item5}).

We have the following result,
\begin{lemma}[Hyperbolic part of Hypotheses $(\mathcal{H}1)$ implies the hyperbolic part of  $(\mathcal{H}2)$]\label{h1h2equiv}
A system \eqref{system} verifying Hypotheses $(\mathcal{H}1)$ parts \eqref{item11},\eqref{item21}, \eqref{item61} will also verify  Hypotheses $(\mathcal{H}2)$ parts \eqref{item2},\eqref{item1}, \eqref{item5}.
\end{lemma}

\begin{figure}[tb]
      \includegraphics[width=.6\textwidth]{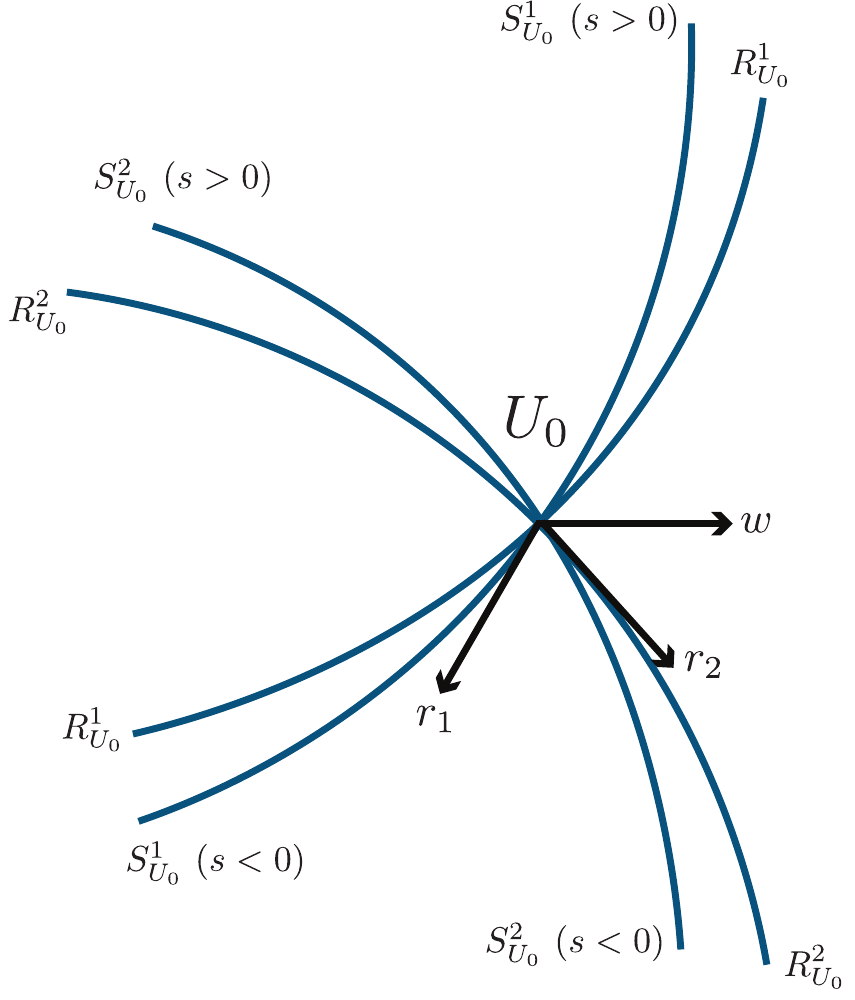}
  \caption{An illustration of \Cref{geo_facts} (based on \cite[Figure 4]{K-K}).}\label{geo_curves_diag}
\end{figure}

Before we prove \Cref{h1h2equiv}, we collect some facts about the geometry of the Hugoniot locus:

\begin{lemma}[\protect{Geometry of the Hugoniot locus \cite[Theorems 4.4-4.7, Theorem 5.1, Theorem 5.4]{K-K}}]\label{geo_facts}

\vspace{.1in}
For an admissible system (\Cref{admsystem}) verifying the sector condition (\Cref{sector_def}):
\begin{enumerate}[(i)]
    \item (global shock curves) For $U_0\in\mathcal{V}$, the Hugoniot locus $H(U_0)$ is exactly the union of two simple arcs $S^1_{U_0}$ and $S^2_{U_0}$ extending to infinity, each of which cuts the $U$-plane into two pieces and contains $U_0$. The shock curves, as they are part of the Hugoniot locus, cannot intersect each other except at $U_0$. For $k=1,2$, we can parameterize $S^k_{U_0}$ as follows: $S^k_{U_0}=S^k_{U_0}(s)$ where for $s>0$ the curves leave $U_0$ in the $-r_k$ direction (corresponding to decreasing $\lambda_k$). We can also choose a smooth parameterization for the $\sigma$, i.e. $\sigma^1_{U_0}$ and $\sigma^2_{U_0}$ such that
\begin{align}
    \sigma^k_{U_0}(s)(S^k_{U_0}(s)-U_0)=f(S^k_{U_0}(s))-f(U_0),
\end{align}
for $k=1,2$. \label{prop0}
    \item The shock curve $S^k_{U_0}$, $k=1,2$, makes third-order contact at $U_0$ with the corresponding rarefaction curve $R^k_{U_0}$, i.e. $S^k$ and $R^k$ have the same tangent and curvature at $U_0$. Recall from \eqref{kappa_formula} that the $R^k$ curves are strictly convex and the $R^1$ curves bend toward $-l_2$ (away from $r_2$), while the $R^2$ curves bend toward $l_1$ (and $r_1$). \label{prop1}
    \item  For $s>0$, $S^k$ lies entirely inside $R^1$ and entirely outside $R^2$. \label{prop2}
    \item  For $s<0$, $S^k$ lies entirely outside $R^1$ and entirely inside $R^2$. \label{prop3}
    \item  The point $U\in S^k (U_0)([0,\infty))$ if and only if $U_0\in S^k (U)((-\infty,0])$. \label{propsymmetry}
    \item  (the Liu entropy condition) For $s\neq 0$, $\frac{d}{ds} \sigma^{k}_{U_0} <0$.  \label{prop4}
    \item  The curves $S^k_{U_0}$ are star-shaped with respect to $U_0$. \label{prop5}
    \item (the Lax E-condition) For $s>0$, $\lambda_k(S^{k}_{U_0}(0))>\sigma^{k}_{U_0}(s)>\lambda_k(S^{k}_{U_0}(s))$.\label{prop6}
    \item For $s<0$, $\lambda_k(S^{k}_{U_0}(0))<\sigma^{k}_{U_0}(s)<\lambda_k(S^{k}_{U_0}(s))$.\label{prop7}
    \item For $s\in\mathbb{R}$, define $T\coloneqq S^{k}_{U_0}(s)-U_0$. Then $w_1\cdot T>0$ (where $w_1$ is from \Cref{sector_def}) when $k=1$ and $s>0$ or $k=2$ and $s<0$. Moreover, $w_1\cdot T<0$ when $k=2$ and $s>0$ or $k=1$ and $s<0$. \label{prop10}
\end{enumerate}
\end{lemma}
\begin{remark}

\hspace{.1in}

\begin{itemize}
    \item See \Cref{geo_curves_diag} for an illustration of the geometric facts from \Cref{geo_facts}.
    \item The Lemma does not require the sector condition. It holds also under the weaker half-plane condition -- see \cite[p.~449]{K-K}.
    \end{itemize}
\end{remark}

We can now prove \Cref{h1h2equiv}.

\begin{proof}[Proof of \Cref{h1h2equiv}]
\hspace{.1in}

\uline{Hypotheses $(\mathcal{H}2)$ \eqref{item2}}: Most of Hypotheses $(\mathcal{H}2)$ \eqref{item2} follows from \Cref{geo_facts} \eqref{prop0}. What is left to show is the continuous dependence of the shock curve $S^k_{U_0}$ on the point $U_0$. 

From \Cref{geo_facts} \eqref{prop0}, we know that  for any admissible system verifying the sector condition, 
 the Hugoniot locus is exactly the union of two simple arcs $S^1_{U_0}$ and $S^2_{U_0}$ extending to infinity, each of which cuts the $U$-plane into two pieces and contains ${U_0}$. 

We now show the continuous dependence of the shock curve $S^k_{U_0}$ on the point $U_0$. Choose $U\neq U_0$ on the curve. We then have the Rankine-Hugoniot condition
\begin{align}
    \bar{\sigma}(U-U_0)=f(U)-f(U_0),
\end{align}
for some $\bar{\sigma}\in\mathbb{R}$.

We appeal to the Implicit Function Theorem. 

From Hypotheses $(\mathcal{H}1)$ \eqref{item61} and \Cref{geo_facts} \eqref{prop6} (the Lax E-condition) and \eqref{prop7}, we have that the matrix $Df(U)-\bar\sigma$ is invertible. Thus, locally around the point $U$, we have a parameterization of the curve $S^k_{\tilde{U}_0}$ for small $\abs{U_0-\tilde{U}_0}$. The parameterization parameter is the shock speed.

\uline{Hypotheses $(\mathcal{H}2)$ \eqref{item1}}: Follows from \Cref{geo_facts} \eqref{prop4}.

\uline{Hypotheses $(\mathcal{H}2)$ \eqref{item5}}: Firstly, remark that due to \Cref{geo_facts} part \eqref{prop1} and part \eqref{prop5}, the curves $S^k_{U}$, for all $s$ and for all $U$, do not cross their tangent line at $U$. Note also \Cref{geo_facts} part \eqref{prop10}, and the fact that the $S^1$ curves cannot cross the $S^2$ curves (\Cref{geo_facts} part \eqref{prop0}). Thus, for each $U$, the union of the images of the $S^1_{U}$ and $S^2_{U}$ shock curves will only intersect a straight line going through $U$ at most twice (including the point $U$ itself). 

\end{proof}

\subsection{Ruling out rank-one connections}\label{sec:norankone}

Given an entropy $\eta$, for $a,b\in\mathcal{V}$ we can define the relative entropy,
\begin{align}\label{rel_entropy}
    \eta(a|b)\coloneqq \eta(a)-\eta(b)-\nabla\eta(b)\cdot(a-b).
\end{align}

Note that if $\eta$ is convex, then $\eta(a|b)\geq 0$. Moreover, if $\eta$ is strictly convex then $\eta(a|b)=0 \iff a=b$.

To rule out rank-one connections we will use the following Lemma, which gives an estimate on the amount of entropy which is dissipated along a shock.

\begin{lemma}[Lax's entropy dissipation formula] \label{diss_lemma}
\hspace{.1in}

For $k=1,2$ and for any shock $({u_L},S_{u_L}^k(s),\sigma_{u_L}^k(s))$,
\begin{align}\label{diss_form}
q(S_{u_L}^k(s))-\sigma_{u_L}^k(s)\eta(S_{u_L}^k(s))=q(u_L)-\sigma_{u_L}^k(s)\eta(u_L) +\int\limits_0^s \frac{\text{d}}{\text{ds}}\big[\sigma_{u_L}^k\big](\tau)\eta({u_L}|S_{u_L}^k(\tau))\,d\tau.
\end{align}
\end{lemma}
\begin{remark}

\hspace{.05in}

\begin{itemize}
    \item \Cref{diss_lemma} can be dated back to Lax \cite{MR0393870}. The result is in fact general and holds for $n\times n$ systems as well. \Cref{diss_lemma} follows immediately from the Rankine-Hugoniot jump condition \eqref{RHcond}. A simple proof is given in \cite{Leger2011}. For the reader's convenience, a proof is given in the Appendix (\Cref{proof_app}). 

    \item Intuitively, for genuinely nonlinear systems, a shock of the $k^{\text{th}}$ characteristic family dissipates entropy because the characteristic speed $\lambda_k(U_L)$ will be different from the characteristic speed $\lambda_k(S_{u_L}^k(s))$, and so intuitively, at a shock,  characteristics from the left are colliding with characteristics from the right,  thus causing loss of information, and an increase of physical entropy (or a decrease of mathematical entropy, as encoded in \eqref{entropy_ineq} and \eqref{entropicshock}). Compare this with the Lax E-condition (\Cref{geo_facts} \eqref{prop6}).
    \end{itemize}
\end{remark}

We can now state and prove a slightly more precise version of \Cref{norankone}.

\begin{proposition}
Consider any system \eqref{system} with global shock curves and verifying the Liu entropy condition (Hypotheses $(\mathcal{H}2)$ \eqref{item2} and \eqref{item1}) and endowed with a strictly convex entropy $\eta$ with associated  entropy-flux $q$. Then the set \eqref{manifold_K} does not contain any rank-one connections.
\end{proposition}
\begin{remark}
By \Cref{h1h2equiv}, the Proposition covers all systems verifying either Hypotheses $(\mathcal{H}1)$ or Hypotheses $(\mathcal{H}2)$.
\end{remark}
\begin{proof}
Due to the strict convexity of $\eta$, the relative entropy \eqref{rel_entropy} will be nonnegative. Thus, the result follows immediately from \Cref{diss_lemma}.
\end{proof}

\subsection{Change of coordinates}

The Hypotheses $(\mathcal{H}1)$ and Hypotheses $(\mathcal{H}2)$ do not cover all interesting systems. But we have this extra freedom: we can change coordinates. In particular, we can go from the Eulerian to the Lagrangian perspective (or vice versa). This allows us to consider a new class of systems. We have the following result.

\begin{theorem}\label{changeofCoordsnoT4}
Assume that the conservation law \eqref{system} with flux $f$ verifies the non-perverse Hugoniot locus condition (\Cref{nonperverse}), is endowed with a strictly convex entropy $\eta$, with associated entropy-flux $q$,  and after a change of coordinates (either from Eulerian to Langrangian or vice-versa, see \Cref{LEcoordinates}), the transformed system verifies either Hypotheses $(\mathcal{H}2)$ or Hypotheses $(\mathcal{H}1)$.  Then, the set $\mathcal{K}_{f,\eta,q}$ cannot contain a $T_4$ configuration..
\end{theorem}
\begin{remark}
For the convenience of the reader, in the Appendix \Cref{equivappendix} we reproduce in \Cref{LEcoordinates} nearly verbatim the main theorem on the equivalence of weak solutions for the  Eulerian and Langrangian equations of gas dynamics from \cite[Theorem 2]{wagnergasdynamics}. As a consequence of \Cref{LEcoordinates} the Rankine-Hugoniot condition and thus the structure of shock sets is preserved under such a transformation. Moreover, level sets of the entropy and entropy-flux are preserved. Due to the nature of our proofs of \Cref{main_theorem1} and \Cref{main_theorem}, these properties of the transformation allow our proofs to still go through. Other transformations are also possible, including transformations of both space and time -- for a discussion of this see \cite[p.~134]{wagnergasdynamics}.
\end{remark}

The proof of \Cref{changeofCoordsnoT4} is in \Cref{sec:proofchangeofCoordsnoT4}.

\subsection{Comparison with Lorent-Peng}\label{sec:lorentpengcomp}

In contrast to Lorent-Peng \cite{MR4144350} and Johansson-Tione \cite{https://doi.org/10.48550/arxiv.2208.10979}, our methods are completely geometric and thus less rigid. In particular, our techniques do not depend on the choice of Lagrangian or Eulerian coordinates for the system under consideration. Thus, we are also able to show $\mathcal{K}_{f,\eta,q}$ does not contain $T_4$ configurations for a large class of systems, including the isentropic Euler system in Eulerian coordinates for a large class of  pressure laws.

However, we find the result of Lorent-Peng \cite{MR4144350}, while algebraic, has some hidden geometric structure which we were able to generalize for the present paper. For example, \cite[Lemma 22]{MR4144350} and \cite[Lemma 23]{MR4144350} are the algebraic equivalent, in the case of the $p$-system, for the geometric argument for Case 1 in our proof of \Cref{main_theorem1} (see \Cref{case1}). We also utilize \Cref{TN_change_sign} in a similar spirit to \cite{MR4144350}.

And in the particular case of the $p$-system, \cite[Lemma 21]{MR4144350} is equivalent to our use of the Smoller-Johnson condition \eqref{SJcond} to eliminate possible cases in the proof of our \Cref{main_theorem} by calculating the total possible number of critical points of $\tilde q $ in the regions $\Romanbar{I}$ and $\Romanbar{IV}$ (see the proof for more details). See also the proof of our \Cref{psystemthird}.

Due to the geometric nature of the casework in the proof of \Cref{main_theorem1}, there is hope our techniques may provide insight into showing nonexistence of $T_N$ configurations for $N>4$. We also highlight that the $T_4$ nonexistence test we utilize in this paper (\Cref{TN_change_sign}) is in fact a nonexistence test for general $T_N$ configurations.

\section{Systems verifying Hypotheses $(\mathcal{H}1)$ or $(\mathcal{H}2)$}\label{sec:systemsverify}





\subsection{A direct application of Hypotheses $(\mathcal{H}1)$}\label{sec:dirapp}

We consider the following large class of systems of conservation laws. Given $\eta\colon \mathbb{R}^2\to\mathbb{R}$ strictly convex, we can write the conservation law
\begin{align}
    \begin{cases}\label{generalpsystem}
    \partial_t v+\partial_x \eta_u(v,u)=0,\\
    \partial_t u +\partial_x \eta_v(v,u)=0,
    \end{cases}
\end{align}
with an entropy given by $\eta$ itself, and entropy-flux $q(v,u)=\eta_u(v,u)\eta_v(v,u)$. As we will show, this is a wide generalization of the $p$-system, in particular allowing for entropy functionals which have nonvanishing mixed partial derivatives.

For this system, the characteristic speeds are given by
\begin{equation}
\begin{aligned}\label{generalpcharspeeds}
    \lambda_1(v,u)=\eta_{vu}-\sqrt{\eta_{uu}\eta_{vv}},\\
    \lambda_2(v,u)=\eta_{vu}+\sqrt{\eta_{uu}\eta_{vv}},
\end{aligned}
\end{equation}
Note that due to the strict convexity of $\eta$, $\eta_{uu}\eta_{vv}>0$ and thus \eqref{generalpcharspeeds} tells us that the system \eqref{generalpsystem} is always strictly hyperbolic.

Moreover, the right eigenvectors (of the Jacobian of the flux) are given by 
\begin{equation}
    \begin{aligned}\label{eigenvectorsgeneralp}
    r_1(v,u)&= \begin{bmatrix}
       -\sqrt{\frac{\eta_{uu}}{\eta_{vv}}}\\[0.3em]
       -1 
     \end{bmatrix},\hspace{.3in}
       r_2(v,u)&= \begin{bmatrix}
       \sqrt{\frac{\eta_{uu}}{\eta_{vv}}}\\[0.3em]
       -1 
     \end{bmatrix},
\end{aligned}
\end{equation}
and the left eigenvectors are given by 
\begin{equation}
    \begin{aligned}
    l_1(v,u)&= \begin{bmatrix}
       -1 & -\sqrt{\frac{\eta_{uu}}{\eta_{vv}}}
     \end{bmatrix},\hspace{.3in}
      l_2(v,u)&= \begin{bmatrix}
       1 & -\sqrt{\frac{\eta_{uu}}{\eta_{vv}}}
     \end{bmatrix},
\end{aligned}
\end{equation}
where for simplicity we have not normalized them.

A direct computation shows that the system \eqref{generalpsystem} is genuinely nonlinear when
\begin{align}\label{gennonlincond}
    \frac{1}{\sqrt{\eta_{vv}}}\Big[\eta_{vvv}\eta_{uu}+3\eta_{vv}\eta_{uuv}\Big]\pm \frac{1}{\sqrt{\eta_{uu}}}\Big[\eta_{vv}\eta_{uuu}+3\eta_{uu}\eta_{vvu}\Big] \neq 0,
\end{align}
where this must hold for both the $+$ and $-$ in the case of the $\pm$ used in \eqref{gennonlincond}.

Another direct computation shows that the system \eqref{generalpsystem} verifies the Smoller-Johnson condition \eqref{SJcond} when
\begin{align}\label{SJcondgeneralp}
\pm \frac{\eta_{uu}}{\eta_{vv}}\eta_{vvu}+\sqrt{\frac{\eta_{uu}}{\eta_{vv}}}\eta_{vuu}\mp\eta_{uuu}-\Big(\frac{\eta_{uu}}{\eta_{vv}}\Big)^{\frac{3}{2}}\eta_{vvv} >0,
\end{align}
where when $+$ is chosen in $\pm$, $-$ must be chosen in $\mp$ and vice-versa (to verify the Smoller-Johnson condition \eqref{SJcondgeneralp} must hold for both of these cases).

\begin{figure}[tb]
      \includegraphics[width=.3\textwidth]{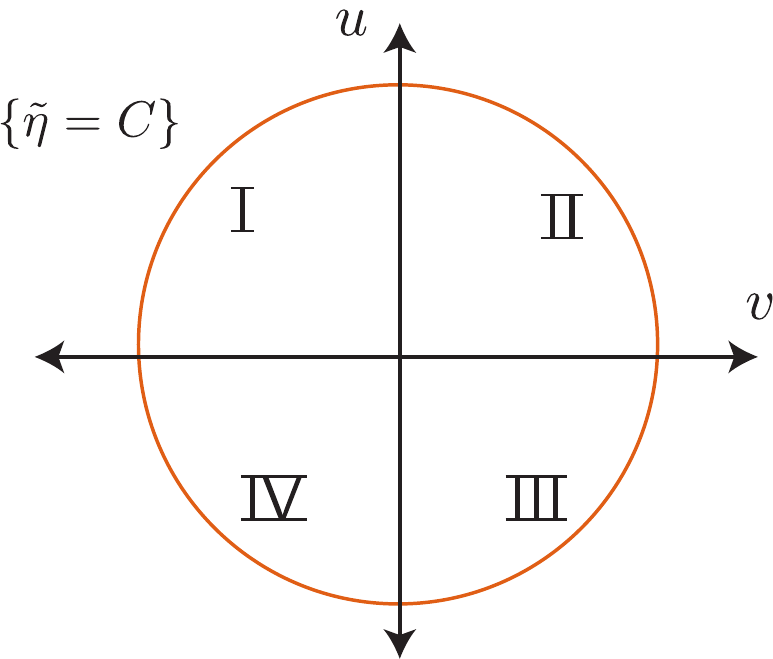}
  \caption{From the proof of \Cref{generalpmeetsconds}, an illustration of the decomposition of the level set of $\tilde\eta$ into the pieces $\Romanbar{I},\ldots,\Romanbar{IV}$.}\label{psystemdecompfig}
\end{figure}

\begin{lemma}\label{generalpmeetsconds}
Assume the system \eqref{generalpsystem} is  genuinely nonlinear, i.e. \eqref{gennonlincond} holds, and verifies the Smoller-Johnson condition, i.e. \eqref{SJcondgeneralp} holds. Further, assume that the characteristic speeds \eqref{generalpcharspeeds}  satisfy $\lambda_1(v,u) \leq 0 \leq \lambda _2(v,u)$ for all $v,u$. Then the system \eqref{generalpsystem} verifies Hypotheses $(\mathcal{H}1)$.
\end{lemma}
\begin{remark}
In general, without the assumptions in Lemma \ref{generalpmeetsconds}, the system  \eqref{generalpsystem} does not have to satisfy Hypotheses $(\mathcal{H}1)$.
\end{remark}
\begin{proof}[Proof of \Cref{generalpmeetsconds}]

Hypotheses $(\mathcal{H}1)$ \eqref{item11}, \eqref{item61}, \eqref{item41} are immediate.

Remark that the  $r_1$ eigenvector always has positive slope, while the $r_2$ eigenvector always has negative slope (see \eqref{eigenvectorsgeneralp}). Thus, \eqref{generalpsystem} verifies the sector condition (\Cref{sector_def}) with $w_1=(1,0)$ and $w_2=(0,-1)$ parallel to the coordinate axes. This shows Hypotheses $(\mathcal{H}1)$ \eqref{item21}. 

Given the entropy $\eta$ and entropy-flux $q$, consider $\tilde \eta$ and $\tilde q$ (see \eqref{tdef}). Remark that due to the lack of assumptions on $\eta$ and $q$, we can also assume without loss of generality that in fact $c_1=c_2=0$ in \eqref{tdef}.

We now show Hypotheses $(\mathcal{H}1)$ property \eqref{item51}. Consider the level set $\{\tilde\eta=C\}$ where $C\in\mathbb{R}$ is not a global minimum of $\tilde \eta$. 

Using the vectors $w_1$ and $w_2$ we decompose the level set $\{\tilde\eta=C\}$ as in \eqref{canonical_decomp}, see \Cref{psystemdecompfig}. 
Then, locally, for each $(v_0,u_0)\in\Romanbar{I}^{\mathrm{o}}\cup\Romanbar{II}^{\mathrm{o}}\cup\Romanbar{III}^{\mathrm{o}}\cup\Romanbar{IV}^{\mathrm{o}}$ (away from points with vertical tangents), we can write the level set $\{\tilde\eta=C\}$ locally as a graph of a function, and talk about the slope of this graph. More precisely, the Implicit Function Theorem gives a real-valued function $g$ defined locally around $v_0$ such that $\tilde\eta(v,g(v))=C$ for $v$ near $v_0$. Furthermore, the derivative is given by
\begin{align}
    g'(v)=\frac{\tilde\eta_v(v,g(v))}{\tilde\eta_u(v,g(v))}.
\end{align}
Remark that the sign of $g'(v)$ is the same as the sign of $\tilde q=\tilde\eta_v\tilde\eta_u$.

At the points on $\{\tilde\eta=C\}$ with horizontal or vertical tangents, $\tilde q$ must be zero by continuity. If any two of the sets $\Romanbar{I}^{\mathrm{o}}$, $\Romanbar{II}^{\mathrm{o}}$, $\Romanbar{III}^{\mathrm{o}}$ and $\Romanbar{IV}^{\mathrm{o}}$ share a boundary point, we can make a choice about which of the two sets the boundary point should belong to. We make the choice in such a way to ensure that Hypotheses $(\mathcal{H}1)$ property \eqref{item51} holds.

\end{proof}

\subsection{$p$-system}

Consider the $p$-system, written (following Smoller \cite{MR1301779})
\begin{align}\label{psystem}
    \begin{cases}
    \partial_t v-\partial_x u=0,\\
    \partial_t u +\partial_x p(v)=0,
    \end{cases}
\end{align}
for $t>0$ and $x\in\mathbb{R}$ and where $p'<0$ and $p''>0$. The natural entropy is 
\begin{align}
    \eta(v,u)&\coloneqq \frac{u^2}{2}-\int\limits^{v}p(s)\,ds,
\shortintertext{and the associated entropy-flux is}
q(v,u)&\coloneqq up(v).
\end{align}

For the $p$-system, we can take $\mathcal{V}\coloneqq \mathbb{R}^{2}$. The $p$-system is in Lagrangian coordinates.

Remark that the $p$-system is a special case of \eqref{generalpsystem}, but with flux $(-\eta_u,-\eta_v)$ in place of $(\eta_u,\eta_v)$ (thus, switching the sign in \eqref{SJcondgeneralp}). We conclude from \eqref{generalpcharspeeds} that the $p$-system with $p'<0$ is strictly hyperbolic. From \eqref{gennonlincond}, we conclude that with $p''>0$ or $p''<0$, the $p$-system is genuinely nonlinear. And lastly, from \eqref{SJcondgeneralp} (and remembering to switch the sign), we conclude that the $p$-system verifies the Smoller-Johnson condition when $p''<0$. When $p''>0$, the inequality in the Smoller-Johnson condition is flipped, i.e. $l_j D^2 f(r_i ,r_i)<0$. 

Thus, from \Cref{generalpmeetsconds} and \Cref{main_theorem1} we conclude that 
\begin{corollary}
    When strictly hyperbolic, i.e. $p'<0$, and genuinely nonlinear, i.e. $p''\neq0$, the $p$-system does not allow for $T_4$ configurations in its constitutive set.
\end{corollary}

Even though this corollary applies to the $p$-system in general, in order to gain intuition for Hypotheses $(\mathcal{H}2)$ (in particular item \eqref{item3}), we now show that in certain cases the $p$-system also satisfies Hypotheses $(\mathcal{H}2)$.

\begin{lemma}\label{psystemthird}
The $p$-system \eqref{psystem} with  $p',p''<0$ and $p'''>0$ verifies Hypotheses $(\mathcal{H}2)$.
\end{lemma}
\begin{proof}
Recall from Lemma \ref{h1h2equiv} that the hyperbolic part of Hypotheses $(\mathcal{H}2)$ (items \eqref{item2},\eqref{item1}, \eqref{item5}) follow from those of Hypotheses $(\mathcal{H}1)$ (items \eqref{item11},\eqref{item21}, \eqref{item61}). Therefore, in light of the Corollary above, it suffices to check item \eqref{item3} from Hypotheses $(\mathcal{H}2)$.

Consider the level set $\{\tilde\eta=C\}$ for some $C\in\mathbb{R}$. We decompose this level set into four pieces  $\Romanbar{I}$, $\Romanbar{II}$, $\Romanbar{III}$ and $\Romanbar{IV}$ as in \eqref{decomp_K_def}. Remark that for the $p$-system the eigenvector $r_1$ always has a strictly positive slope, while the eigenvector $r_2$ always has a strictly negative slope. Thus we can choose $\Romanbar{I}$, $\Romanbar{II}$, $\Romanbar{III}$ and $\Romanbar{IV}$ such that the slope of $\Romanbar{I}$ and $\Romanbar{III}$ is always nonnegative, and the slope of $\Romanbar{II}$ and $\Romanbar{IV}$ is always nonpositive.

For a fixed $C\in\mathbb{R}$, we have the following parameterizations of the level sets  $\{\tilde\eta=C\}$,
\begin{align}\label{params}
    u=-c_2\pm\sqrt{c_2^2-2(-\int\limits^v p(s)\,ds+c_1 v-C)}.
\end{align}

Note that from \eqref{params}, along the ``top'' (take the $+$ in \eqref{params}) of $\{\tilde\eta=C\}$, the slope of the level set is negative when $(p(v)-c_1)<0$. Similarly, along the ``bottom'' of $\{\tilde\eta=C\}$, the slope of the level set is negative when $(p(v)-c_1)>0$.

We remark that the following analysis of critical points is similar to what will come next in the proof of \Cref{main_theorem1}.

We want to find critical points of $\tilde q$ restricted to a level set of the function $\tilde \eta$. This is a constrained optimization problem, and we use Lagrange multipliers. Thus, $U\in\mathcal{V}$ is a critical point if and only if
\begin{align}
    \nabla \tilde q(U) =\lambda \nabla\tilde\eta (U).
\end{align}
for some $\lambda\in\mathbb{R}$. Due to the compatibility condition $\nabla\tilde q= \nabla\tilde \eta Df$ between $\tilde \eta$ and $\tilde q$, we have $\nabla\tilde \eta(U) Df(U)=\lambda \nabla\tilde\eta (U)$. Thus, $\nabla\tilde\eta(U)$ is a left eigenvector of $Df$.

Note then that due to $l_ir_j=0$ for $i\neq j$, the critical points of $\tilde q$ restricted to a level set of the function $\tilde \eta$ occur when one of the eigenvector fields $r_i$ is parallel with the level set of $\tilde\eta$.

Note that 
\begin{align} 
    r_1^\top \nabla^2\tilde q r_1&=(u+c_2)p''(v) +2p'(v)\sqrt{-p'(v)} \label{secondr1}\\
    r_2^\top \nabla^2\tilde q r_2&=(u+c_2)p''(v) -2p'(v)\sqrt{-p'(v)}\label{secondr2},
\end{align}

Then, from \eqref{secondr2} it is clear that $\tilde q$ will have at most one critical point in $\Romanbar{IV}$. This is because, by definition of the set $\Romanbar{IV}$, the $r_1$ vector will always be transverse to the boundary of the level set of $\tilde\eta$. Hence, critical points will occur only at points of $\Romanbar{IV}$ where the level set of $\tilde\eta$ is parallel to $r_2$. And we know that at one of these critical points, $\tilde q$ will be convex (when restricted to the level set of $\tilde\eta$).  Furthermore, we cannot have a function which is strictly convex at two adjacent critical points. Thus, we conclude there is at most one critical point in $\Romanbar{IV}$. This critical point, if it exists, will be a minimum. 

Likewise, due to \eqref{secondr1}, $\tilde q$ will have at most one critical point in $\Romanbar{I}$. This critical point, if it exists, will be a maximum.

Thus, if $\tilde q$ has a critical point in $\Romanbar{III}$, at least one such critical point will be at a point where $r_1^\top \nabla^2\tilde q r_1 \leq 0$. Recall, we cannot have a function which is strictly convex at two adjacent critical points.

Note also, in $\Romanbar{III}$
\begin{align}
    \frac{\textrm{d}}{\textrm{d}v}\Big[ r_1^\top \nabla^2\tilde q r_1 \Big]
    =\frac{\textrm{d}u}{\textrm{d}v}p''(v)+(u+c_2)p'''(v)+3\sqrt{-p'(v)}p''(v) < 0,
\end{align}
where we view $u$ as a function of $v$ (according to \eqref{params}), and we have used \eqref{params} to show that $u+c_2\leq 0$ in $\Romanbar{III}$.

Thus, once we hit a critical point in $\Romanbar{III}$ such that $r_1^\top \nabla^2\tilde q r_1 \leq 0$, we have $r_1^\top \nabla^2\tilde q r_1 < 0$ for all other potential critical points in $\Romanbar{III}$. Remark that as above, we cannot have a function which is strictly concave at two adjacent critical points.

Similarly, in $\Romanbar{II}$
\begin{align}
    \frac{\textrm{d}}{\textrm{d}v}\Big[ r_2^\top \nabla^2\tilde q r_2 \Big]
    =\frac{\textrm{d}u}{\textrm{d}v}p''(v)+(u+c_2)p'''(v)-3\sqrt{-p'(v)}p''(v) > 0.
\end{align}

Thus $\tilde q$ has at most one extremal critical point in each of the four pieces $\Romanbar{I}$, $\Romanbar{II}$, $\Romanbar{III}$ and $\Romanbar{IV}$. Moreover, if the level set of $\tilde\eta$ is unbounded, at least one of these four pieces will be empty. This shows Hypotheses $(\mathcal{H}2)$ \eqref{item3}.
\end{proof}
\begin{remark}
 Instead of using the direct computation \eqref{secondr1} and \eqref{secondr2}, we can use the Smoller-Johnson condition (see \eqref{SJcond}) to give a general proof of the number of critical points in $\Romanbar{I}$ and $\Romanbar{IV}$. In fact, we do this in the proof of \Cref{main_theorem1} (see \Cref{sec:proofmain_theorem1}).
\end{remark}

\subsection{Isentropic Euler}

We also consider, in Eulerian coordinates, isentropic Euler,
\begin{align}\label{isentropicEulergeneral}
     \begin{cases}
    \partial_t \rho-\partial_x (\rho v)=0,\\
    \partial_t (\rho v) +\partial_x [\rho v^2 +P(\rho)]=0.
    \end{cases}
\end{align}

We consider a large class of pressure laws $P$. We assume $P\colon [0,\infty)\to\mathbb{R}$ is smooth. The condition $P'(\rho)>0$ ensures the system \eqref{isentropicEulergeneral} is strictly hyperbolic, and $[\rho P(\rho]''\neq 0$ ensures the system is genuinely nonlinear.  The vacuum state is $\rho=0$. We work away from vacuum, and only consider $\rho>0$.

The characteristic speeds for this system are 
\begin{align}
    \lambda_1(\rho,\rho v)=v-\sqrt{P'(\rho)}, \hspace{.3in} \lambda_2(\rho,\rho v)=v+\sqrt{P'(\rho)}.
\end{align}

The corresponding right eigenvectors are given by, 

\begin{equation}
    \begin{aligned}\label{eigenvectorsisentropicEuler}
    r_1(\rho,\rho v)&= \begin{bmatrix}
       -1\\[0.3em]
       -v+\sqrt{P'(\rho)} 
     \end{bmatrix}, \hspace{.3in}
       r_2(\rho,\rho v)&= \begin{bmatrix}
       1\\[0.3em]
       v+\sqrt{P'(\rho)} 
     \end{bmatrix},
\end{aligned}
\end{equation}
where for simplicity we have not normalized them. Observe that the system \eqref{isentropicEulergeneral} does not verify the sector condition (\Cref{sector_def}) and, moreover, does not verify the sign condition on the eigenvalues required by Hypotheses $(\mathcal{H}1)$ \eqref{item61}.

The system \eqref{isentropicEulergeneral} is endowed with a natural entropy, entropy-flux pair,
\begin{equation}
\begin{aligned}\label{entropyisentropicgeneral}
    \eta&=\frac{1}{2}\rho v^2 +S(\rho),\\
    q&= \frac{1}{2}\rho v^3+\rho u S'(\rho),
\end{aligned}
\end{equation}
where $S$ verifies $S''(\rho)=\rho^{-1} P'(\rho)>0$.

The system \eqref{isentropicEulergeneral} obtains its canonical form by changing the state variables from $(\rho,v)$ to $(\rho,m)$, where $m=\rho v$ is the momentum density. In the canonical variables $(\rho,m)$, the entropy $\eta$ is strictly convex.

We have as a special case of \eqref{isentropicEulergeneral}, the rectilinear isentropic flow of an ideal gas. Following Dafermos \cite[p.~230]{dafermos_big_book}, we write
\begin{align}\label{isentropicEuler}
     \begin{cases}
    \partial_t \rho-\partial_x (\rho v)=0,\\
    \partial_t (\rho v) +\partial_x [\rho v^2 +\kappa \rho^\gamma]=0,
    \end{cases}
\end{align}
for a constant $\kappa>0$, and for $\gamma>1$.  For $\gamma>1$, the system \eqref{isentropicEuler} is strictly hyperbolic and genuinely nonlinear. 

The system \eqref{isentropicEuler} is endowed with the following entropy, entropy-flux pair,
\begin{align}
    \eta&=\frac{1}{2}\rho v^2 +\frac{\kappa}{\gamma-1}\rho^\gamma,\\
    q&= \frac{1}{2}\rho v^3+\frac{\kappa\gamma}{\gamma-1}\rho^\gamma v.
\end{align}

Remark that taking $\gamma=2$ in \eqref{isentropicEuler} gives the \emph{system of shallow water waves.}

\vspace{.1in}

As remarked earlier, the system \eqref{isentropicEulergeneral} does not verify Hypotheses $(\mathcal{H}1)$. However, after changing from Eulerian to Langrangian coordinates, the transformed system does verify Hypotheses $(\mathcal{H}1)$. Thus, as a corollary of \Cref{changeofCoordsnoT4}, we have the following result:

\begin{corollary}\label{CORchangeofCoordsnoT4}
When the system \eqref{isentropicEulergeneral} is strictly hyperbolic, i.e. $P'(\rho)>0$, and genuinely nonlinear, i.e. $[\rho P(\rho]''\neq 0$, the constitutive set for this equation $\mathcal{K}_{f,\eta,q}$ (with entropy, entropy-flux given by \eqref{entropyisentropicgeneral}) does not contain $T_4$ configurations.
\end{corollary}
\begin{proof}
Our goal is to apply \Cref{changeofCoordsnoT4}.

First, we check that the non-perverse Hugoniot locus condition (\Cref{nonperverse}) is verified. 

For simplicity, let us work for the moment in $(\rho,v)$ space and not $(\rho,m)$ space.

If $\{\sigma,(\rho_L,v_L),(\rho_R,v_R)\}$ verify the Rankine-Hugoniot jump condition, then the following relation holds
\begin{align}
   (v_L-v_R)^2=\frac{(P(\rho_R)-P(\rho_L))(\rho_R-\rho_L)}{\rho_R\rho_L}.
\end{align}

Thus, for a fixed point in state space $(\rho_L,v_L)$, possible other points in state space which can  be connected to this point via a shock have the form $(\rho_R,v_R(\rho_R))$ where $\rho_R\in[0,\infty)$ and
\begin{align}\label{shockcurveisentropic}
    v_R(\rho_R)=v_L\pm \sqrt{\frac{(P(\rho_R)-P(\rho_L))(\rho_R-\rho_L)}{\rho_R\rho_L}}.
\end{align}

Due to the hyperbolicity assumption, \eqref{shockcurveisentropic} says that the two shock curves are strictly monotonic: we have the relation $\pm \sgn(\rho_R-\rho_L){v'}_{R}(\rho_R)>0$ \cite[p.~293]{Leger2011}. Furthermore, from \eqref{shockcurveisentropic} we see that the four curves given by $S^k_{(\rho_L,v_L)}$ for $k=1,2$ and for $s<0$ and $s>0$ each live in separate quadrants of $(\rho,v)$ space, determined by two lines parallel to the coordinate axes and which  go through the point $(\rho_L,v_L)$.

Moreover, genuine nonlinearity ensures that the Lax E-condition and Liu entropy condition are verified. In fact, for the system \eqref{isentropicEulergeneral}, genuine nonlinearity is equivalent to a strictly monotone shock speed along each shock curve\footnote{For more details on the Lax E-condition, Liu entropy condition, and other calculations involving the shock curves of the isentropic Euler system, many references are available. See for example \cite[p.~293]{Leger2011} and references therein.}.

Thus, the non-perverse Hugoniot locus condition is verified. We note that this  argument is similar to the proof of \cite[Theorem 5.4]{K-K}.

Remark also that the in the canonical variables $(\rho,m)$, the entropy  \eqref{entropyisentropicgeneral} is strictly convex.

We now check that when we convert the system \eqref{isentropicEulergeneral} to Langrangian coordinates, the resulting system meets the Hypotheses $(\mathcal{H}1)$.

From \Cref{LEcoordinates}, we get the corresponding Langrangian system is
\begin{align}\label{isentropicEulerconvert}
     \begin{cases}
    \partial_t v_1 -\partial_y v_2 =0,\\
    \partial_t v_2 +\partial_y P(v_1^{-1})=0.
    \end{cases}
\end{align}
Thus, we receive the $p$-system (see \eqref{psystem}).

We have also the corresponding entropy, entropy-flux
\begin{align}
    \hat\eta(V)&=\frac{v_2^2}{2}+S(v_1^{-1})v_1,\\
    \hat q (V)&=v_2 \big(\frac{S'(v_1^{-1})}{v_1}-S(v_1^{-1})\big).
\end{align}

It is immediate to check that when the system \eqref{isentropicEulergeneral} is strictly hyperbolic, i.e. $P'(\rho)>0$, then \eqref{isentropicEulerconvert} is also strictly hyperbolic. Likewise, when the system \eqref{isentropicEulergeneral} is genuinely nonlinear, i.e. $[\rho P(\rho]''\neq 0$, then \eqref{isentropicEulerconvert} is also genuinely nonlinear.

The compatibility condition $\nabla \hat{q}=\nabla\hat\eta D(-v_2,P(v_1^{-1}))$ implies that $\hat q$ can be written in the form $\hat{q} = \hat\eta_{v_1}\hat\eta_{v_2}$. However, note that we do not necessarily have that $P(v_1^{-1})=\eta_{v_1}$.

Thus, Hypotheses $(\mathcal{H}1)$ \eqref{item51} follows from (the proof) of \Cref{generalpmeetsconds}.

The result follows from \Cref{changeofCoordsnoT4}.

\end{proof}

\subsection{Two copies of Burgers}

To conclude this section, we consider two copies of Burgers, coupled only at the level of the entropy and entropy-flux:
\begin{align}\label{twoBurgers}
    \begin{cases}
  \partial_t u_1 + \partial_x f_1(u_1)=0,\\
  \partial_t u_2 + \partial_x f_2(u_2)=0,
    \end{cases}
\end{align}
for smooth functions $f_i\colon\mathbb{R}\to\mathbb{R}$. The system is strictly hyperbolic when the images of $f_1'$ and $f_2'$ are disjoint, and the system is genuinely nonlinear when $f_i''\neq 0$ for $i=1,2$. The system \eqref{twoBurgers} admits the following entropy, entropy-flux pair,
\begin{align}
    \eta(u_1,u_2)&=h(u_1)+g(u_2),\\
    q(u_1,u_2)&=\int\limits^{u_1} h'(s)f_1'(s)\,ds+\int\limits^{u_2} g'(s)f_2'(s)\,ds,
\end{align}
for any smooth and strictly convex $h$ and $g$.

When the system \eqref{twoBurgers} is strictly hyperbolic, this exhausts the set of all entropy, entropy-flux pairs for the system due to the compatibility condition $\nabla q=\nabla\eta D(f_1,f_2)$ and using that $\nabla q$ must be curl-free. 
 
 Remark that \eqref{twoBurgers} does not verify the Smoller-Johnson condition \eqref{SJcond}; its rarefaction curves are in fact straight lines parallel to the coordinate axes. Thus it does not satisfy Hypotheses $(\mathcal{H}1)$. Furthermore, we place no restrictions on the values of $f_1'$ and $f_2'$ -- they might not verify one nonpositive, one nonnegative, thus failing Hypotheses $(\mathcal{H}1)$ \eqref{item61}.

However, the straight rarefaction curves can actually help us:
\begin{lemma}\label{twocopiesH2}
Two copies of Burgers \eqref{twoBurgers}, when strictly hyperbolic and genuinely nonlinear, verifies Hypotheses $(\mathcal{H}2)$.
\end{lemma}
\begin{proof}
Remark that the rarefaction curves (and shock curves) are straight lines parallel to the coordinate axes.

For a reference showing that Burgers satisfies the Lax E-condition, see \cite[p.~275]{dafermos_big_book}. For the Liu entropy condition, see \cite[p.~279]{dafermos_big_book}.

Thus Hypotheses $(\mathcal{H}2)$ \eqref{item4}, \eqref{item2}, \eqref{item1} and \eqref{item5} all immediately follow. 

We now show Hypotheses $(\mathcal{H}2)$ \eqref{item3}.

 We want to find critical points of $\tilde q$ restricted to a level set of the function $\tilde \eta$. This is a constrained optimization problem, and as we do elsewhere in this paper, we use Lagrange multipliers. Thus, $U\in\mathbb{R}^{2}$ is a critical point if and only if
\begin{align}
    \nabla \tilde q(U) =\lambda \nabla\tilde\eta (U).
\end{align}
for some $\lambda\in\mathbb{R}$. Due to the compatibility condition $\nabla\tilde q= \nabla\tilde \eta D(f_1,f_2)$ between $\tilde \eta$ and $\tilde q$, we have $\nabla\tilde \eta(U) D(f_1,f_2)(U)=\lambda \nabla\tilde\eta (U)$. Thus, $\nabla\tilde\eta(U)$ is a left eigenvector of $D(f_1,f_2)$.

Remark that we have the following relation between right eigenvectors $r_1,r_2$ and left eigenvectors $l_1,l_2$,
\begin{align}\label{norm_vectBurgers}
l_i r_j = \begin{cases}
0 &\text{if } i \neq j,   \\
1 &\text{if } i=j.   \end{cases}
\end{align}

We now count the maximum number of times that the vector fields $r_i$ of right eigenvectors are parallel with a given level set of $\tilde\eta$. By using \eqref{norm_vectBurgers}, this will allow us to count critical points.

Then, due to the strict convexity of sublevel sets $\{\tilde\eta<C\}$ and the fact that the rarefaction curves are parallel to the coordinate axes, we have at most two critical points corresponding to each characteristic family. Thus, we have at most four critical points in total. Remark that if the level set of $\tilde\eta$ is unbounded, then there will be less than four critical points. This shows Hypotheses $(\mathcal{H}2)$ \eqref{item3}.
\end{proof}

\section{Proof that Hypotheses Imply Nonexistence of $T_4$ in Eulerian or Lagrangian Coordinates}\label{sec:proofshyp}

\subsection{Preliminaries}

We state the following simple Lemma. It plays an important role in our proofs in \Cref{sec:proofmain_theorem1} and \Cref{sec:proofmain_theorem}.

\begin{lemma}[The shock curve perturbation Lemma]\label{curve_perturb}
Let $\alpha\colon\mathbb{R}\to\mathcal{V}$ be a continuous curve parameterized by $t\in\mathbb{R}$.

Consider the system \eqref{system} with Hypotheses $(\mathcal{H}1)$ or Hypotheses $(\mathcal{H}2)$. Consider the four connected components of $U$-space which are determined by the Hugoniot locus at the point $\alpha(0)$. If a point $U\in\mathcal{V}$ is in one of these four connected components, then it will be in the same connected component for all $t$ unless there exists $t_0$ such that the Hugoniot locus at $\alpha(t_0)$ intersects $U$.
\end{lemma}
\begin{proof}
Under Hypotheses $(\mathcal{H}2)$, this follows from the map $(U_L,s)\mapsto S^k_{u_L}(s)$ being continuous (for $k=1,2$) (see Hypotheses $(\mathcal{H}2)$ \eqref{item2}). For Hypotheses $(\mathcal{H}1)$, see \Cref{h1h2equiv}.  
\end{proof}

We now introduce the following  geometric-linear algebra Lemma.

\begin{lemma}[Geometric-linear algebra Lemma]\label{lin_combo}

Consider a system \eqref{system} endowed with a strictly convex entropy $\eta$ and associated entropy-flux $q$ and verifying the non-perverse Hugoniot locus condition (\Cref{nonperverse}).

Assume there are four points $\{X_1,\ldots,X_4\}\subset \mathcal{K}_{f,\eta,q} \subset \mathbb{R}^{3\times2}$ in $T_4$ configuration (which in particular implies no rank-one connections pairwise).

Suppose $(P,C_i,\kappa_i)$ is the parameterization of $\{X_i\}$ corresponding to \eqref{T_N_param}, in other words $(P,C_i,\kappa_i)$ is a solution to the equations \eqref{T_N_param} with the left-hand side given by $\{X_i\}$.  

Write $C_i=a_i\otimes n_i$ for $a_i\in\mathbb{R}^{3\times1},n_i\in\mathbb{R}^{1\times2}$, for $i=1,\ldots,4$.

Define the matrix $B\in\mathbb{R}^{3\times8}$  by concatenating the four elements of the $T_4$. More precisely,
\begin{align}
    B\coloneqq
     \begin{bmatrix} 
    \uparrow & \uparrow &  \uparrow &\uparrow \\
    X_1 & X_2  & X_3 & X_4 \\
    \downarrow & \downarrow  & \downarrow & \downarrow
    \end{bmatrix}.
\end{align}

Similarly, define the concatenation of four copies of $P$,
\begin{align}
    \Pi\coloneqq
     \begin{bmatrix} 
    \uparrow & \uparrow &  \uparrow &\uparrow \\
    P & P  & P & P \\
    \downarrow & \downarrow  & \downarrow & \downarrow
    \end{bmatrix} \in\mathbb{R}^{3\times8}.
\end{align}

Then, we conclude the third row of $B-\Pi$ is a linear combination of the first two rows.
\end{lemma}
\begin{proof}
We have $\sum_{i} C_i=0$. We can rewrite this as
\begin{align}\label{rewrite_sum}
\sum_i (a_i)_{j,1}(n_i)_{1,k}=0
\end{align}
for all $j=1,2,3$ and $k=1,2$.

\uline{Step 1}

We first show that $\text{rank} (B-\Pi)\leq 2$.

Define the matrix
\begin{align}
A\coloneqq \begin{bmatrix} 
    (n_1)_{1,1} & (n_2)_{1,1} &  (n_3)_{1,1} &(n_4)_{1,1} \\
    (n_1)_{1,2} & (n_2)_{1,2} &  (n_3)_{1,2} &(n_4)_{1,2}
    \end{bmatrix}.
\end{align}

Note that if $\text{rank} A=1$, then there exists $\bar{\sigma}\in\mathbb{R}$ such that every element of the $T_4$ has a first column which is a $\bar{\sigma}$ multiple of the second column. This implies there are rank-one connections in the $T_4$, which is in contradiction with the assumption of non-degeneracy of the $T_4$.

Thus $A$ has full row rank and $\text{rank} A=2$. This implies that the dimension of the null space of $A$ is 2.

Fix $j\in\{1,2,3\}$ and consider the map $\Gamma_j$ from the null space of $A$ to a subset of $\mathbb{R}^{1\times8}$, defined by
\begin{align}
    [(a_1)_{j,1} \hspace{.2in} (a_2)_{j,1} \hspace{.2in} (a_3)_{j,1} \hspace{.2in} (a_4)_{j,1}] \mapsto \mbox{the $j^{\text{th}}$ row of $B-\Pi$}.
\end{align}

Because the domain of $\Gamma_j$ is the null space of $A$, \eqref{rewrite_sum} holds for $k=1,2$. 

Remark that by the definition of the $X_i$, the $\Gamma_j$ map is linear. Further, the null space of $A$ has dimension 2. Thus, by the rank-nullity theorem we can conclude that the image of $\Gamma_j$ is a two dimensional subset of $\mathbb{R}^{1\times8}$. Furthermore, the image of $\Gamma_j$ is the same for all $j$.  

We can conclude that $\text{rank} (B-\Pi)\leq 2$.

\uline{Step 2}

We now show by contradiction that the last row of $(B-\Pi)$ is a linear combination of the first two rows.

Let us write 
\begin{align}
    U_i\coloneqq ((X_i)_{1,1},(X_i)_{2,1}).
\end{align}

Assume that the first two rows of $B-\Pi$ are linear multiples of each other. This implies that for every $i\neq j$ the points $U_i$ and $U_j$  (for $i,j=1,\ldots,4$) in the state space are connected by a shock.

Furthermore, when the first two rows of $B-\Pi$ are linear multiples of each other, then the $U$-space points  $U_i$ (for $i=1,\ldots,4$) lie on a straight line. Thus, from the non-perverse Hugoniot locus condition (\Cref{nonperverse}), we can conclude that there is a fixed $k\in\{1,2\}$ such that the curve $S^k_{U_i}$ contains the points $U_j$ for all $i$ and $j$. 

Then, let $I$ be a value of $i\in\{1,2,3,4\}$ which maximizes the quantity
\begin{align}
    \lambda_k (U_i).
\end{align}

Consider then the curve $S^k_{U_I}$. By above, we can conclude that this curve contains the points $U_i$ for all $i$. Furthermore, from the non-perverse Hugoniot locus condition (\Cref{nonperverse}), we know the shock speed function $\sigma^k_{U_I}$ verifies the Lax E-condition \eqref{laxeHyP} and the Liu entropy condition. Note as well the fact that $\sigma^k_{U_I}(0)= \lambda_k (U_I)$ (which also follows from the Lax E-condition). Hence, we can conclude that $U_i=S^k_{U_I}(s_i)$ (for all $i\neq I$), for $s_i\in\mathbb{R}$ and where the $s_i$ \emph{all have the same sign.}

Then, we get a contradiction due to \Cref{TN_change_sign}. 

More precisely, consider the determinant of the $2\times2$ matrix $X_I-X_i$ (with the middle row deleted):

\begin{align}\label{lincalc}
&((X_I)_{1,1}-(X_i)_{1,1}) (q(U_I)-q(U_i))
-(\eta(U_I)-\eta(U_i))(f_1(U_I)-f_1(U_i))
\\
&\hspace{.7in}= ((X_I)_{1,1}-(X_i)_{1,1})\Big[(q(U_I)-q(U_i))
-\sigma^k(s_i)(\eta(U_I)-\eta(U_i))\Big],\label{samesignresult}
\end{align}
where the last line comes from the fact that the speed $\sigma^k(s_i)$ of the shock connecting $U_I$ and $U_i$ is given by
\begin{align}\label{whatifdenom}
\frac{(f_1(U_I)-f_1(U_i))}{((X_I)_{1,1}-(X_i)_{1,1})}.
\end{align}

Then, due to the strict convexity of $\eta$ (and thus the the non-negativity of the relative entropy), the Liu entropy condition, the fact that the $s_i$ all have the same sign, and the monotonicity of the coordinates of the shock curve $S^k_{U_I}$ (see the non-perverse Hugoniot locus condition (\Cref{nonperverse})), we can conclude that \eqref{samesignresult} has the same sign for all $i$, thus contradicting the existence of a $T_4$ due to \Cref{TN_change_sign}.

Lastly, remark that if the denominator in \eqref{whatifdenom} is zero, then the second row of $X_I-X_i$ is a zero-multiple of the first row, and the computation of the $2\times2$ determinant should be repeated with the first row deleted instead of the second row (the computation is nearly identical). Similarly, if the $((X_I)_{1,1}-(X_i)_{1,1})$ terms do not have the same sign for all $i$, the computation of the $2\times2$ determinant should be repeated with the first row deleted instead of the second row.

We conclude that the last row of $(B-\Pi)$ is a linear combination of the first two rows.
\end{proof}

\subsection{Proof of \Cref{main_theorem1}}\label{sec:proofmain_theorem1}

We now give the argument for the nonexistence of $T_4$  for a $2\times2$ system of conservation laws verifying the Hypotheses $(\mathcal{H}1)$.

We argue by contradiction.

Let the four points of a $T_4$ be $ Y_1,\ldots, Y_4\in \mathcal{K}_{f,\eta,q}$.

By \Cref{h1h2equiv}, we have a non-perverse Hugoniot locus. By \Cref{lin_combo}, we know that there exists constants $c_1,c_2\in\mathbb{R}$ such that the system \eqref{system}, with the entropy $\tilde \eta$ and entropy-flux $\tilde q$ (see \eqref{tdef}), admits a $T_4$ configuration $X_1,\ldots, X_4\in \mathcal{K}_{f,\tilde\eta,\tilde q}$ with the property that $(X_i)_{3,1}=(X_j)_{3,1}$ and $(X_i)_{3,2}=(X_j)_{3,2}$ for all $j,k$. Remark, the $T_4$ points have corresponding $U$-plane coordinates $((X_1)_{(1,1)},(X_1)_{(1,2)}),\ldots,((X_4)_{(1,1)},(X_4)_{(1,2)})$.

By above, we know the four points are on the same level sets of $\tilde\eta$ and $\tilde q$. In particular, they are on the level set $\{\tilde\eta=C\}$ for some $C\in\mathbb{R}$. At this point in the proof, we no longer use any properties of being a $T_4$ configuration, including the specific ordering of the elements $ Y_1,\ldots, Y_4$ in the definition of $T_4$ configuration (see \Cref{T_N_def}). 

We know that $C$ cannot a global minimum of $\tilde\eta$ because such a minimum would be unique due to the strict convexity of $\eta$ and thus the four $T_4$ points would not be distinct and this would contradict the definition of a $T_4$ configuration.

Thus, as in \eqref{decomp_K_def}, we can write $\{\tilde\eta=C\}=\Romanbar{I}\cup \Romanbar{II} \cup \Romanbar{III} \cup \Romanbar{IV}$.

Due to Hypotheses $(\mathcal{H}1)$ part \eqref{item51} and the fact that all $T_4$ points are on the same level set of $\tilde q$, there are two possible cases: all of the $T_4$ points (in $U$-space) are in $\Romanbar{II}\cup \Romanbar{IV}$ or all of the $T_4$ points are in $\Romanbar{I}\cup \Romanbar{III}$.

We only consider the case when all of the $T_4$ points (in $U$-space) are in $\Romanbar{II}\cup \Romanbar{IV}$. The  case when all of the $T_4$ points  are in $\Romanbar{I}\cup \Romanbar{III}$ is very similar.

We first present two Lemmas which will be used repeatedly.

\begin{lemma}[Growth of $\tilde q$ inside level sets of $\tilde\eta$]\label{ordering_q}
For $U_L\in \{\tilde\eta=C\}$ and $s<0$ such that $S^1_{U_L}(s)\in\{\tilde\eta<C\}$, we have $\tilde q(S^1_{U_L}(s))>\tilde q(U_L)$.

Similarly, for $U_L\in \{\tilde\eta=C\}$ and $s>0$ such that $S^2_{U_L}(s)\in\{\tilde\eta<C\}$, we have $\tilde q(S^2_{U_L}(s))<\tilde q(U_L)$.
\end{lemma}
The proof of \Cref{ordering_q} follows immediately from Hypotheses $(\mathcal{H}1)$ part \eqref{item61}, \Cref{diss_lemma}, as well as the information on the sign of $\sigma^k$ from \Cref{geo_facts} part \eqref{prop6} and part \eqref{prop7}.

\begin{lemma}[Intersection of $S^1$ curves in the level set of $\tilde\eta$]\label{intersection_lemma}
For $U_R, \hat{U}_R\in\{\tilde\eta=C\}$ with the property that $\tilde q (U_R)=\tilde q(\hat{U}_R)$, consider the curves $S^1_{U_R}$ (for $s<0$) and $S^1_{\hat{U}_R}$ (for $s>0$). Then the two curves \emph{cannot intersect} at a point $U_L$ which has the property that $U_L\in\{\tilde\eta<C\}$.
\end{lemma}
\begin{proof}
We argue by contradiction. Assume such a $U_L$ exists. Then, by \Cref{geo_facts} \eqref{propsymmetry} there exists $s>0$ such that $U_R=S^1_{U_L}(s)$ and there exists $\hat{s}<0$ such that $\hat{U}_R=S^1_{U_L}(\hat{s})$.
From two applications of \Cref{diss_lemma}, we write
\begin{align}\label{contra1}
    \tilde q(U_R) -\tilde q(U_L) < \sigma^1_{U_L}(s)(\tilde\eta(U_R)-\tilde\eta(U_L))
\end{align}
for $s>0$ and
\begin{align}\label{contra2}
    \tilde q(\hat{U}_R) -\tilde q(U_L) > \sigma^1_{U_L}(\hat{s})(\tilde\eta(\hat{U}_R)-\tilde\eta(U_L))
\end{align}
for $\hat{s}<0$.

Then, recall that due to $U_L\in\{\tilde\eta<C\}$, $\tilde\eta(\hat{U}_R)-\tilde\eta(U_L)=\tilde\eta(U_R)-\tilde\eta(U_L)>0$. 

Moreover, from Hypotheses $(\mathcal{H}1)$ part \eqref{item61}, and \Cref{geo_facts} part \eqref{prop6} and part \eqref{prop7}, we know $\sigma^1_{U_L}(s)<\sigma^1_{U_L}(\hat{s})<0$. Thus \eqref{contra1} and \eqref{contra2} give a contradiction. This completes the proof.
\end{proof}

To determine how many of the $T_4$ points will lie in $\Romanbar{IV}$, we will calculate the total number of critical points of $\tilde q$ (with $\tilde q$ restricted to $\Romanbar{IV}$). Due to all $T_4$ points being on the same level set of $\tilde q$, this will determine the maximum number of the $T_4$ points that can lie in $\Romanbar{IV}$.

 Our aim is to find critical points of $\tilde q$ restricted to a level set of the function $\tilde \eta$. This is a constrained optimization problem, and we use Lagrange multipliers. Thus, $U\in\mathbb{R}^{2}$ is a critical point if and only if
\begin{align}
    \nabla \tilde q(U) =\lambda \nabla\tilde\eta (U).
\end{align}
for some $\lambda\in\mathbb{R}$. Due to the compatibility condition $\nabla\tilde q= \nabla\tilde \eta Df$ between $\tilde \eta$ and $\tilde q$, we have $\nabla\tilde \eta(U) Df(U)=\lambda \nabla\tilde\eta (U)$. Thus, $\nabla\tilde\eta(U)$ is a left eigenvector of $Df$.

Note then that due to $l_ir_j=0$ for $i\neq j$, the critical points of $\tilde q$ restricted to a level set of the function $\tilde \eta$ occur when one of the eigenvector fields $r_i$ is parallel with the level set of $\tilde\eta$.

Remark that 
\begin{align}
\nabla^2\tilde q = \nabla^2 \tilde\eta Df +\nabla\tilde\eta D^2f.
\end{align}

Thus,  
\begin{align}\label{consequence_SJ}
r_i^\top \nabla^2\tilde q r_i = \lambda_i r_i^\top\nabla^2 \tilde\eta r_i +\nabla\tilde\eta D^2 f(r_i ,r_i),
\end{align}
for $i=1,2$. 

Note that in $\Romanbar{IV}$, at a critical point of $\tilde q$, $\nabla\tilde\eta \parallel l_1$. In fact, $\nabla\tilde \eta \cdot l_1>0$. Thus, from the Smoller-Johnson condition \eqref{SJcond} we get that $\nabla\tilde\eta D^2 f(r_2 ,r_2)>0$. In particular, $r_2^\top \nabla^2\tilde q r_2>0$ due to $\lambda_2>0$ and the strict convexity of $\eta$ (which makes $\nabla^2 \tilde\eta$ positive-definite). Then, it is clear that $\tilde q$ will have at most one critical point in $\Romanbar{IV}$. This is because, by definition of the set $\Romanbar{IV}$, the $r_1$ vector will always be transverse to the boundary of the level set of $\tilde\eta$. Hence, critical points will occur only at points of $\Romanbar{IV}$ where the level set of $\tilde\eta$ is parallel to $r_2$. And we know that at one of these critical points, $\tilde q$ will be convex (when restricted to the level set of $\tilde\eta$).  Furthermore, we cannot have a function which is strictly convex at two adjacent critical points. Thus, we conclude there is at most one critical point in $\Romanbar{IV}$. This critical point, if it exists, will be a minimum.

And similarly, remark $r_1^\top \nabla^2\tilde q r_1<0$ at a possible critical point in $\Romanbar{I}$ and thus it is clear that $\tilde q$ will have at most one critical point in $\Romanbar{I}$. This critical point, if it exists, will be a maximum \footnote{This is useful to know in the case when at least one of the $T_4$ points  is in $\Romanbar{I}$, a case which, as mentioned above, we do not consider here.}.

Remark that if the Smoller-Johnson condition has its sign flipped, i.e. $l_j D^2 f(r_i ,r_i)<0$, we can still prove \Cref{main_theorem1}, but we consider the opposite side of the level set of $\tilde\eta$ (the sets $\Romanbar{II},\Romanbar{III}$).

Label the $T_4$ points $U_1,\ldots,U_4$ (they are marked by the $\times$ symbol in our figures, e.g. \Cref{case1}).

We then consider three cases:

\begin{itemize}
    \item Case 1: Four of the $T_4$ points are in $\Romanbar{II}$, and none are in $\Romanbar{IV}$.
    \item Case 2: Three of the $T_4$ points are in $\Romanbar{II}$, and one is in $\Romanbar{IV}$.
    \item Case 3: Two of the $T_4$ points are in $\Romanbar{II}$, and two are in $\Romanbar{IV}$.
\end{itemize}

Note that due to the existence of only one local extrema of $\tilde q$ in $\Romanbar{IV}$, at most two $T_4$ points can be in $\Romanbar{IV}$. Thus, these three cases are exhaustive.

We now begin the casework. 

\uline{Case 1}

\begin{figure}[tb]
      \includegraphics[width=.9\textwidth]{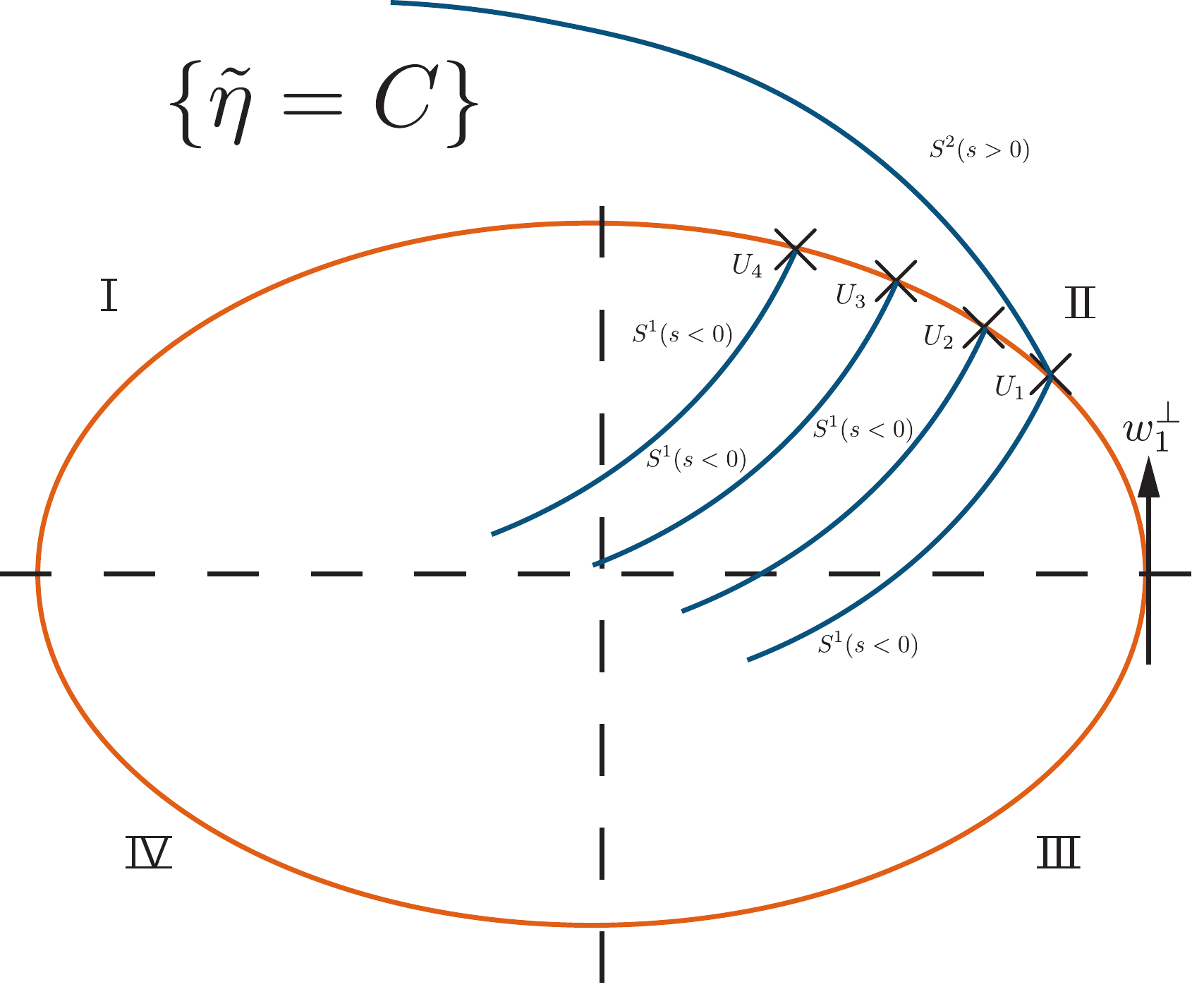}
  \caption{For Case 1: the point $U_1$ gives us a contradiction to the existence of a $T_4$.}\label{case1}
\end{figure}

Consider the case where all four points $U_1,\ldots,U_4$ of the $T_4$ (in $U$-space) are in the set $\Romanbar{II}$ (see \Cref{case1}). 

\begin{claim}
For $U_L\in \Romanbar{II}, S^1_{U_L}(s)$ does not intersect $\Romanbar{II}$ for $s<0$.
\end{claim}
\begin{claimproof}
We argue by contradiction. Assume there exists $s_R<0$ such that $U_R\coloneqq S^1_{U_L}(s_R)\in\Romanbar{II}$. Then we know that there exists $s_L>0$ such that $U_L=S^1_{U_R}(s_L)$. Consider the tangent to the curve $S^1_{U_R}$ at $s=0$. Notice that this tangent line cannot intersect $\Romanbar{II}$ more than once (at $U_R$, in fact), because if it did, then by the Mean Value Theorem and \Cref{geo_facts} part (\ref{prop1}), a point on $\Romanbar{II}$ would have a tangent vector parallel with $r_1(U_R)$, a contradiction to the definition of $\Romanbar{II}$. Recall the sector condition (\Cref{sector_def}).

Recall also that due to \Cref{geo_facts} part \eqref{prop1} and part \eqref{prop5}, the curve $S^1_{U_R}(s)$, for all $s$, does not cross its tangent line at $U_R$. Moreover, the point $U_L$ must be on the \emph{other side} of this tangent line (opposite from the curve $S^1_{U_R}$) because $U_R=S^1_{U_L}(s_R)$ and for all $U\in\Romanbar{II}$, $r_1(U)$ points into the level set of $\tilde\eta$. This gives a contradiction, and shows the claim.
\end{claimproof}

Consider now the $T_4$ point closest to where $\Romanbar{II}$ is tangent to $w^\perp_1$ (the point $U_1$ in \Cref{case1}). By the claim, we know that the $S^1$ curve (for $s<0$) emanating from this point does not cross between any of the other $T_4$ points. Furthermore, by the claim and \Cref{ordering_q}, we know the $S^2$ curve (for $s>0$) emanating from this point also does not cross between any of the other $T_4$ points. Lastly, by \Cref{geo_facts} (\ref{prop0}), we know that the $S^1$ curve (for $s>0$) and the $S^2$ curve (for $s<0$) cannot cross either of the $S^1$ curve (for $s<0$) or the $S^2$ curve (for $s>0$). Thus, we must have that the $U_1$ point gives us a contradiction to the existence of a $T_4$ by \Cref{TN_change_sign}.

\uline{Case 2}

We now consider the case when only three of the $T_4$ points are in $\Romanbar{II}$. The fourth point is in $\Romanbar{IV}$.

\uline{Subcase 2.1} 

Consider the $T_4$ point in $\Romanbar{II}$ closest to where $\Romanbar{II}$ is tangent to $w^\perp_1$ (the point $U_1$ in \Cref{case2}). Consider now the case where the $S^1$ curve (for $s<0$) emanating from  $U_1$ does not cross between any of the other $T_4$ points (see \Cref{case2}). In this case, similar to Case 1, we are done by again  considering the $U_1$ point in $\Romanbar{II}$  and  invoking \Cref{TN_change_sign}. 
\begin{figure}[tb]
      \includegraphics[width=.9\textwidth]{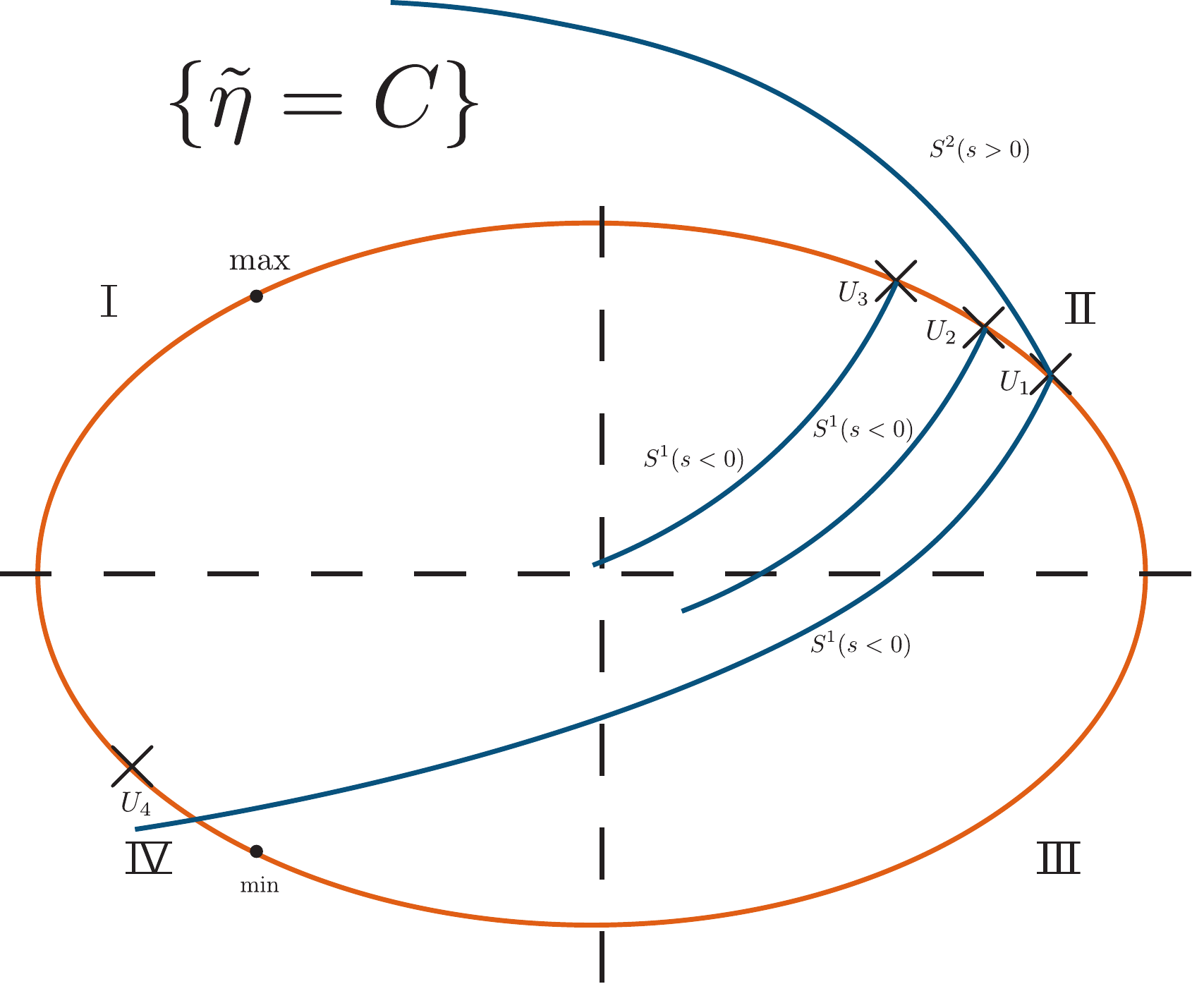}
  \caption{Subcase 2.1: in this case, the $U_1$ point in $\Romanbar{II}$ gives us a contradiction to the existence of a $T_4$.}\label{case2}
\end{figure}

\uline{Subcase 2.2}

\begin{figure}[tb]
      \includegraphics[width=.9\textwidth]{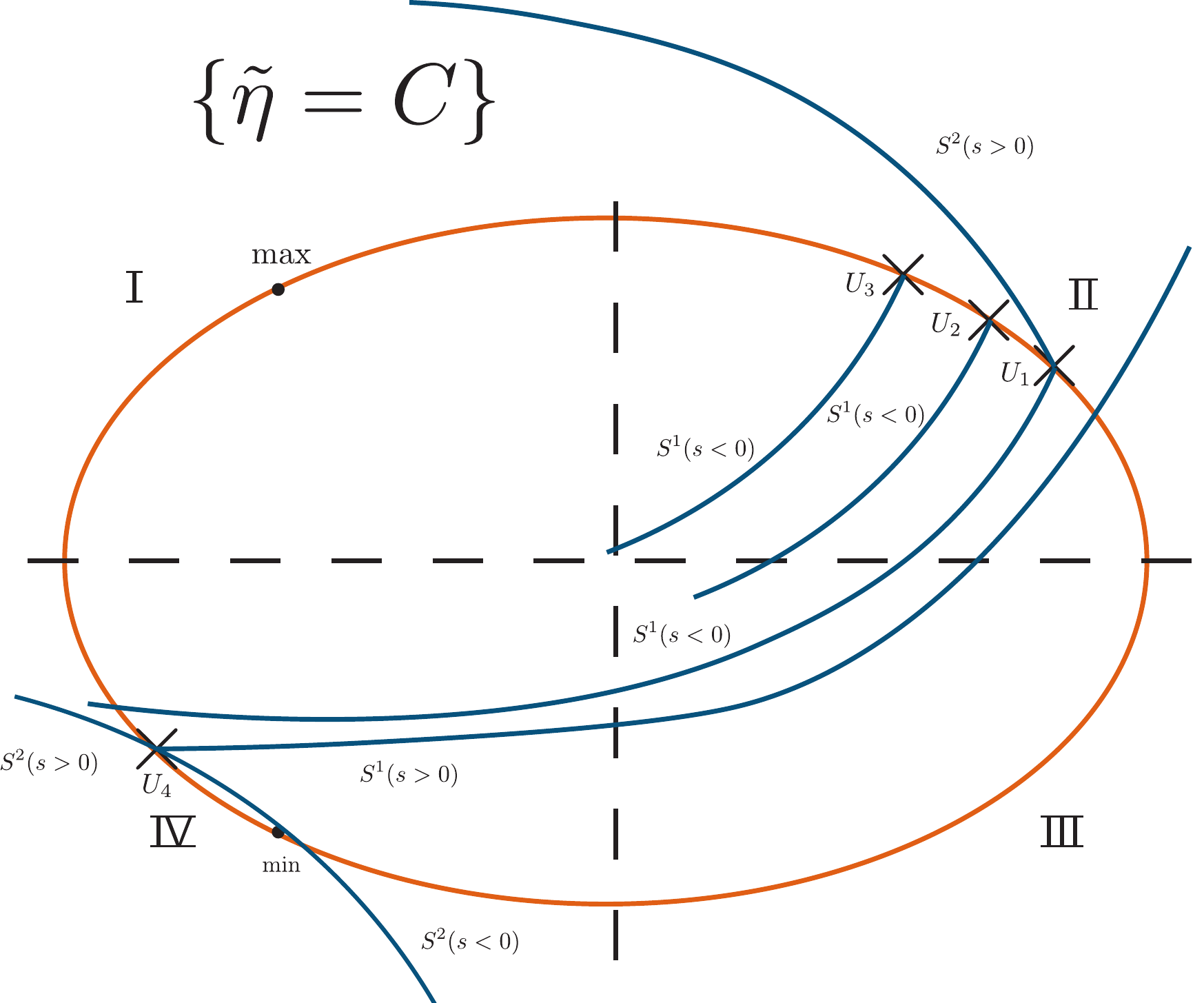}
  \caption{Subcase 2.2: in this case, we get a contradiction to the existence of a $T_4$ by considering the point $U_4$ in $\Romanbar{IV}$.}\label{case3}
\end{figure}

Otherwise, the $S^1$ curve emanating from the $U_1$ point in $\Romanbar{II}$ crosses between the $T_4$ point in $\Romanbar{IV}$ and the set $\Romanbar{II}$ (see \Cref{case3}). But then, in this case, due to \Cref{intersection_lemma}, the $S^1$ curve (for $s>0$) emanating from the $U_4$ point in $\Romanbar{IV}$ does not cross between any other $T_4$ points.

We argue by contradiction to show that the $S^2$ curve (for $s>0$) emanating from the $U_4$ point in $\Romanbar{IV}$ also does not cross between any other two $T_4$ points.

Assume that it does in fact cross between two of the $T_4$ points. We follow the same arguments that will 
be used in the proof of \Cref{main_theorem} (in particular, we use \Cref{curve_perturb}). 

Note that when this $S^2$ curve intersects the level set of $\tilde\eta$, by \Cref{diss_lemma} and \Cref{geo_facts} part \eqref{prop4}, $\tilde q$ at this point of intersection will be less than the value of $\tilde q$ at the $T_4$ points. Then, we can imagine dragging the starting point of this $S^2$ curve from the $T_4$ point  $U_4$ to the local minimum of $\tilde q$ (restricted to the level set of $\tilde\eta$) which is located in $\Romanbar{IV}$. By continuity, the point in the level set of $\tilde\eta$ where the $S^2$ curve intersects is always between two of the $T_4$ points in the level set of $\tilde\eta$ (these two $T_4$ points ``lock'' the shock curve between them). See also \Cref{curve_perturb}. This is a contradiction, because at the local minimum in $\Romanbar{IV}$, the $r_2$ eigenvector is parallel to the level set of $\tilde\eta$. In more detail: recall that the $S^2$ curves and $R^2$ curves have third-order contact (\Cref{geo_facts} part \eqref{prop1}). Furthermore, remark that due to \Cref{geo_facts} part \eqref{prop1} and part \eqref{prop5}, the curves $S^k_{U_0}$, for all $s$, do not cross their tangent line at $U_0$ (recall also that the rarefaction curves $R^k$ are convex and in particular the $R^2$ curves bend toward $l_1$ (and $r_1$)). Here, $U_0$ is in the context of \Cref{geo_facts}. Recall also that the level sets of $\tilde\eta$ are convex. Thus the $S^2$ curve at the local minimum cannot enter the level set of $\tilde\eta$ at all, and this gives the contradiction.

Thus, we can conclude that the three $T_4$ points $U_1,U_2,U_3$ in $\Romanbar{II}$ are between the $S^1$ curve (for $s>0$) and the $S^2$ curve (for $s>0$).

Lastly, as for Case 1, by  \Cref{geo_facts} (\ref{prop0}) we know that the $S^1$ curve (for $s<0$) and the $S^2$ curve (for $s<0$) cannot cross either of the $S^1$ curve (for $s>0$) or the $S^2$ curve (for $s>0$).

Then, by considering the $T_4$ point $U_4$ in $\Romanbar{IV}$, we are done due to \Cref{TN_change_sign}.

We remark that this idea of dragging the starting point of a shock curve and invoking \Cref{curve_perturb} will be used to great effect in the proof of \Cref{main_theorem} (see \Cref{sec:proofmain_theorem}).

\uline{Case 3} 

Let us now consider the case when there are two $T_4$ points in the set $\Romanbar{IV}$.

\uline{Subcase 3.1}

If two points are  in $\Romanbar{IV}$, let us consider the case where the $S^1$ curve (for $s<0$) emanating from the $U_1$ point in $\Romanbar{II}$ does not cross between any other $T_4$ points (see \Cref{case4a}). In this case, as in Case 1 and Subcase 2.1, we are done by considering the $U_1$ point in $\Romanbar{II}$ and invoking \Cref{TN_change_sign}.

\begin{figure}[tb]
      \includegraphics[width=.9\textwidth]{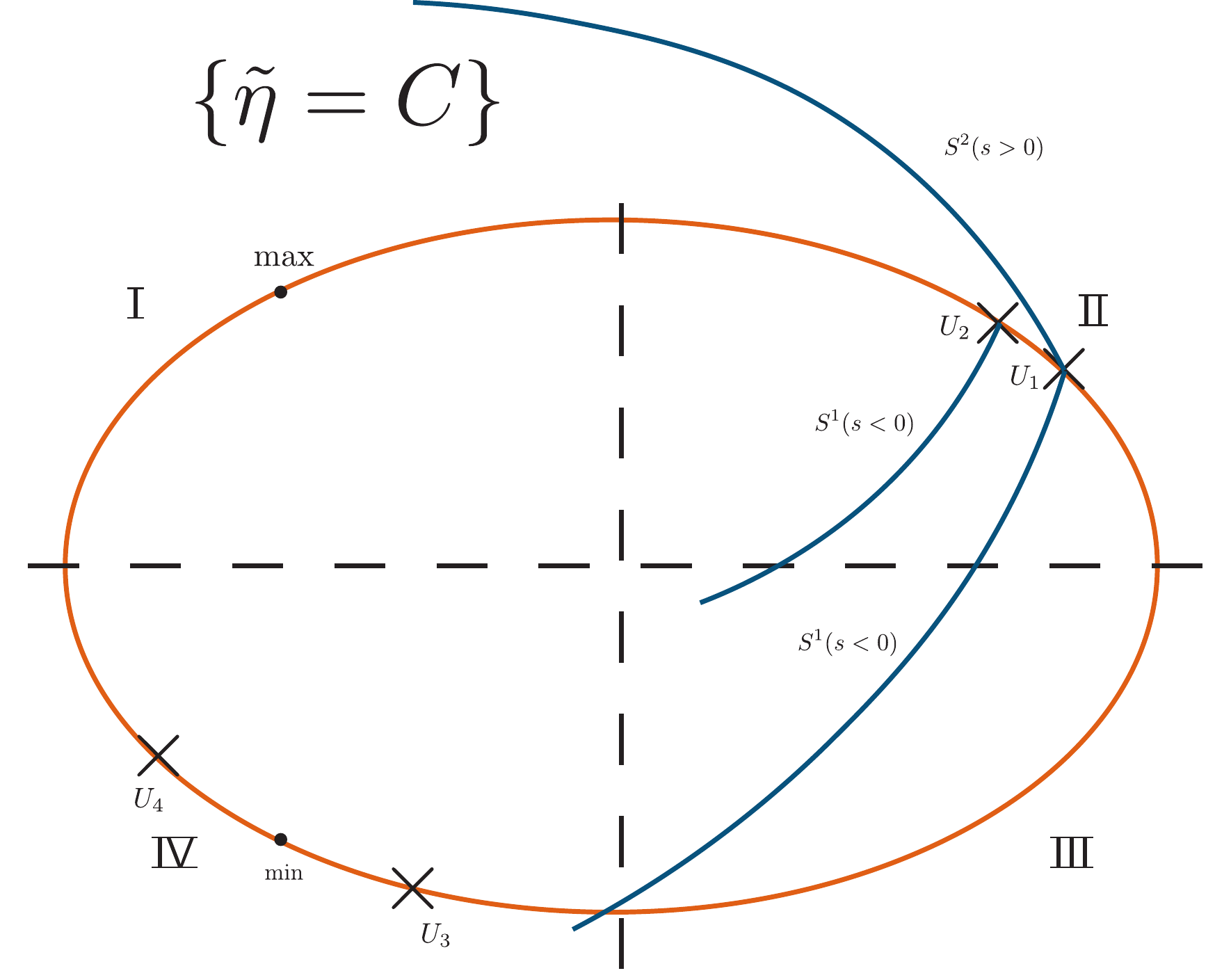}
  \caption{Subcase 3.1: in this case, as in Case 1 and Subcase 2.1, we are done by considering the $U_1$ point in $\Romanbar{II}$.}\label{case4a}
\end{figure}

\uline{Subcase 3.2}
\begin{figure}[tb]
      \includegraphics[width=.9\textwidth]{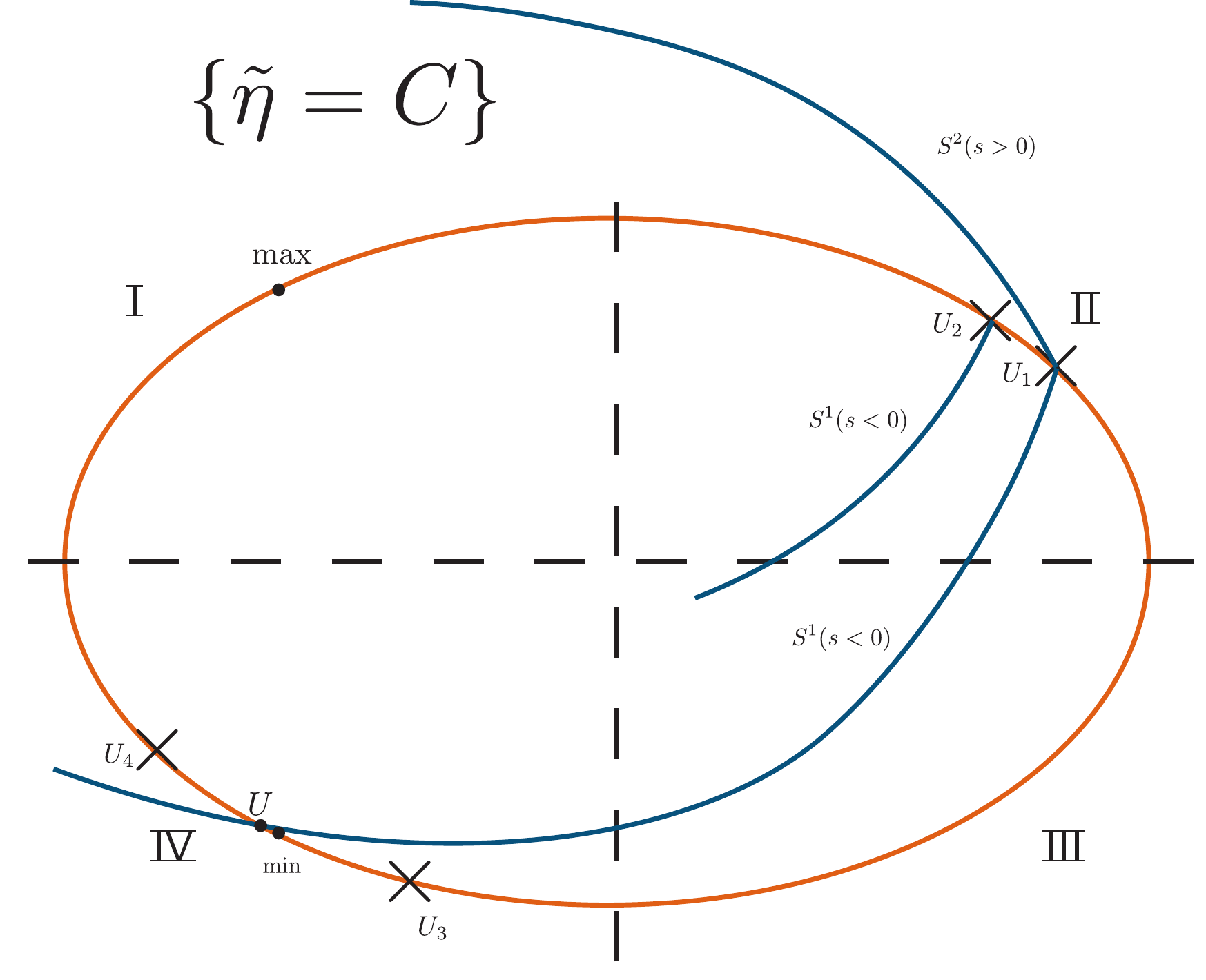}
  \caption{Subcase 3.2: we want to eliminate this possibility for the $S^1$ curve (for $s<0$) emanating from the $U_1$ point in $\Romanbar{II}$.}\label{case32}
\end{figure}

We consider the case when the $S^1$ curve emanating from the $U_1$ point in $\Romanbar{II}$ crosses between the two  $T_4$ points  in $\Romanbar{IV}$ (the points $U_3$ and $U_4$ - see  \Cref{case32}) and intersects the level set of $\tilde\eta$ at a point $U$. However, this cannot occur because at the point $U$, we must have $\tilde q(U)>\tilde q(U_0)$. This is a contradiction because the points $U_1,\ldots, U_4$ are all on the same level set of $\tilde q$ and in between the points  $U_3$ and $U_4$, the function $\tilde q$ (restricted to the level set of $\tilde\eta$) has a single critical point which is a minimum.

\uline{Subcase 3.3} 

\begin{figure}[tb]
      \includegraphics[width=.9\textwidth]{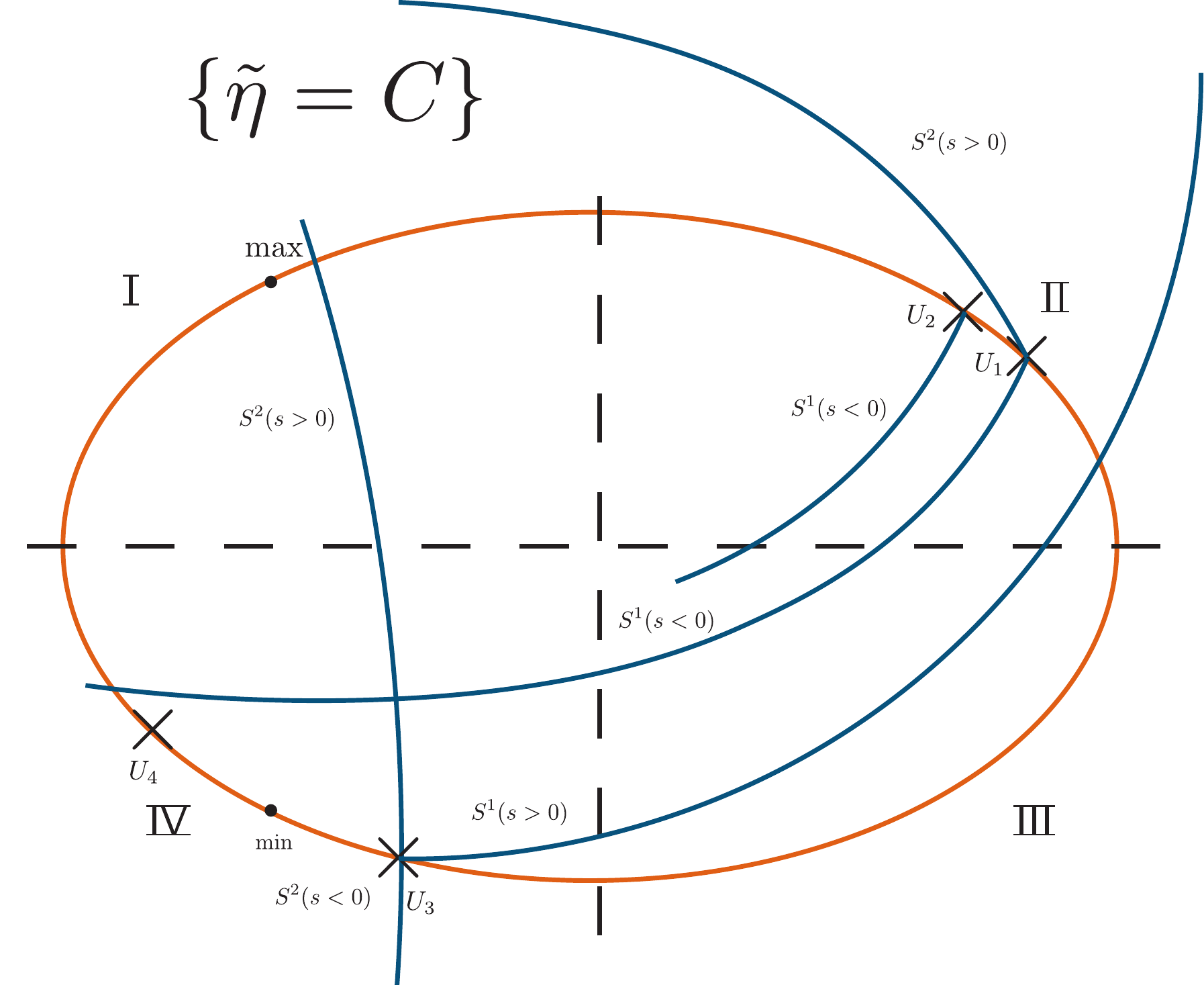}
  \caption{Subcase 3.3: we want to eliminate this possibility for the $S^2$ curve (for $s>0$) emanating from the $U_3$ point in $\Romanbar{IV}$.}\label{case4b}
\end{figure}

Otherwise, the $S^1$ curve emanating from the $U_1$ point in $\Romanbar{II}$ crosses between the  $U_4$ point (in $\Romanbar{IV}$) and the set $\Romanbar{I}$ (see \Cref{case4b}).  In this case, we consider the $T_4$ point in $\Romanbar{IV}$ furthest from the place where $w^\perp_1$ is tangent to $\Romanbar{IV}$ (the $U_3$ point in $\Romanbar{IV}$ -- see \Cref{case4b}). By \Cref{intersection_lemma}, this point must have a $S^1$ curve ($s>0$) which does not cross the $S^1$ curve ($s<0$) emanating from the $U_1$ point in $\Romanbar{II}$. However, as illustrated in \Cref{case4b}, the $S^2$ ($s>0$) curve emanating from the $U_3$ point in $\Romanbar{IV}$ may be problematic. But, as in Subcase 3.1, we can conclude that the other three $T_4$ points  are between the $S^1$ curve (for $s>0$) and $S^2$ curve (for $s>0$) emanating from the $U_3$ 
 point in $\Romanbar{IV}$ (see \Cref{case5}).

Thus, we must be in the case as shown in \Cref{case5}.

This completes the casework.

For an explicit example of the shock curves at each $T_4$ point, see \Cref{t4fig}. This gives an example from the $p$-system. 

\begin{figure}[tb]
      \includegraphics[width=.9\textwidth]{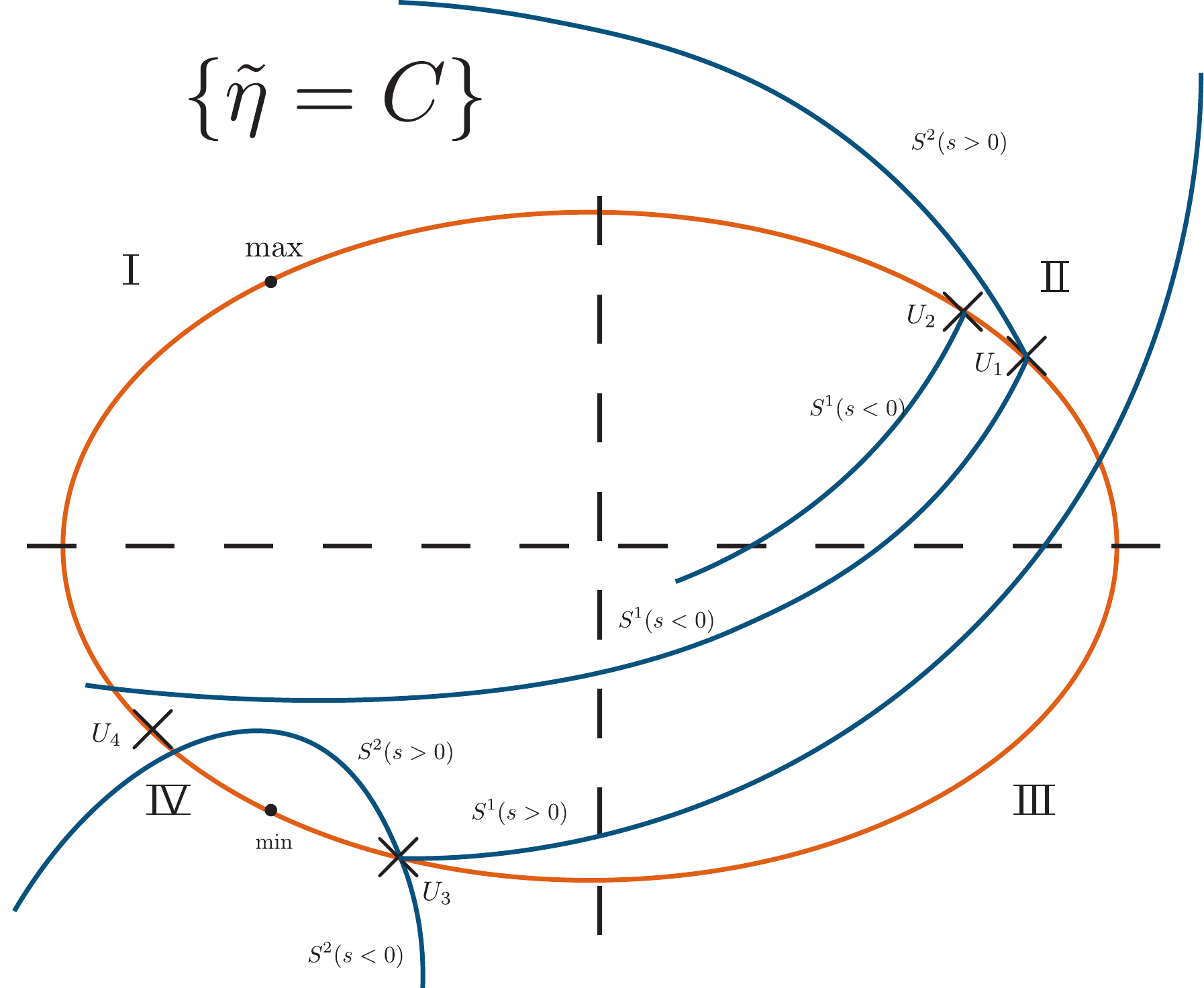}
  \caption{Conclusion of Subcase 3.3: as in Subcase 3.1, we can conclude that three of the $T_4$ points are between the $S^1$ curve (for $s>0$) and the $S^2$ curve (for $s>0$) emanating from the $U_3$ point in $\Romanbar{IV}$.}\label{case5}
\end{figure}

\subsection{Proof of \Cref{main_theorem}}\label{sec:proofmain_theorem}

We now give the argument for the nonexistence of $T_4$  for a $2\times2$ system of conservation laws verifying the Hypotheses $(\mathcal{H}2)$.

We argue by contradiction.

Assume $ Y_1,\ldots, Y_4\in \mathcal{K}_{f,\eta,q}$ is such a $T_4$ configuration.

Recall here that by Hypotheses $(\mathcal{H}2)$ \eqref{item5}, the non-perverse Hugoniot locus condition (\Cref{nonperverse}) is verified. Thus, as in the proof of \Cref{main_theorem1}, by \Cref{lin_combo} we know that there exists constants $c_1,c_2\in\mathbb{R}$ such that the system \eqref{system}, with the entropy $\tilde \eta$ and entropy-flux $\tilde q$ (see \eqref{tdef}), admits a $T_4$ configuration $X_1,\ldots, X_4\in \mathcal{K}_{f,\tilde\eta,\tilde q}$ with the property that $(X_j)_{3,1}=(X_k)_{3,1}$ and $(X_j)_{3,2}=(X_k)_{3,2}$ for all $j,k$. Remark, the $T_4$ points have corresponding $U$-plane coordinates $((X_1)_{(1,1)},(X_1)_{(1,2)}),\ldots,((X_4)_{(1,1)},(X_4)_{(1,2)})$.

Let us write
\begin{align}
    U_i\coloneqq ((X_i)_{1,1},(X_i)_{2,1}),
\end{align}
for $i=1,\ldots,4$.

By above, we know the four points $U_1,\ldots,U_4$ are on the same level sets of $\tilde\eta$ and $\tilde q$. In particular, they are on the level set $\{\tilde\eta=C\}$ for some $C\in\mathbb{R}$. They are also on the level set $\{\tilde q=(X_1)_{3,2}\}$.

It follows from Hypotheses $(\mathcal{H}2)$ (\ref{item3}) that $\{\tilde\eta=C\}$ and $\{\tilde q=(X_1)_{3,2}\}$ intersect in at most four points. If they intersect at fewer points than four points, we have reached a contradiction and the proof is complete. 

We can then assume by Hypotheses $(\mathcal{H}2)$ (\ref{item3}) that $\tilde q$ restricted to a level set of $\tilde\eta$ has exactly two local maxima and exactly two local minima -- otherwise we have reached a contradiction. Remark that this means we are in the case where the level set of $\tilde\eta$ is bounded.

Without loss of generality, we can assume that $U_1$ is the point on $\{\tilde\eta=C\}$ between  global maxima $M\in\{\tilde\eta=C\}$ of $\tilde q$ and global minima $m\in\{\tilde\eta=C\}$ of $\tilde q$ on the set $\{\tilde\eta=C\}$. Note that the global maxima does not have to be unique, i.e. there might be more than one global maxima. Similarly, the global minima does not have to be unique. Consider one of the shock curves $S^k_{U_1}$ (for $k=1$ or $k=2$) emanating from $U_1$. 

If neither of the shock curves $S^1_{U_1}$ or $S^2_{U_1}$ enter the set $\{\tilde\eta < C\}$, then all of the other $U_i$ (for $i=2,3,4$) are in the same connected component of the Hugoniot locus at the point $U_1$ so by \Cref{TN_change_sign} we are done. 

Otherwise, due to the convexity of $\eta$ and Hypotheses $(\mathcal{H}2)$ (\ref{item1}), the integral in \eqref{diss_form} has a sign, and thus for $s\in\mathbb{R}\setminus\{0\}$ such that $S^k_{U_1}(s)\in \{\tilde\eta=C\}$, we must have $\tilde q (S^k_{U_1}(s))\neq \tilde q(U_1)$. In particular, consider $s$ such that $S^k_{U_1}(s)$ is the point where the $k$-shock curve exits the set $\{\tilde\eta\leq C\}$ for the last time. Note that the shock curve must exit the set because by Hypotheses $(\mathcal{H}2)$ (\ref{item3}) we are in the case where the level set is bounded, and by Hypotheses $(\mathcal{H}2)$ (\ref{item2}) the shock curves go to infinity. We have two cases: $\tilde q (S^k_{U_1}(s)) > \tilde q(U_1)$ and $\tilde q (S^k_{U_1}(s)) < \tilde q(U_1)$. We handle the first case (the second case is nearly identical). 

Thus, assume 
\begin{align}\label{ordering_qEq}
\tilde q (S^k_{U_1}(s)) > \tilde q(U_1). 
\end{align}

\begin{figure}[tb]
      \includegraphics[width=.7\textwidth]{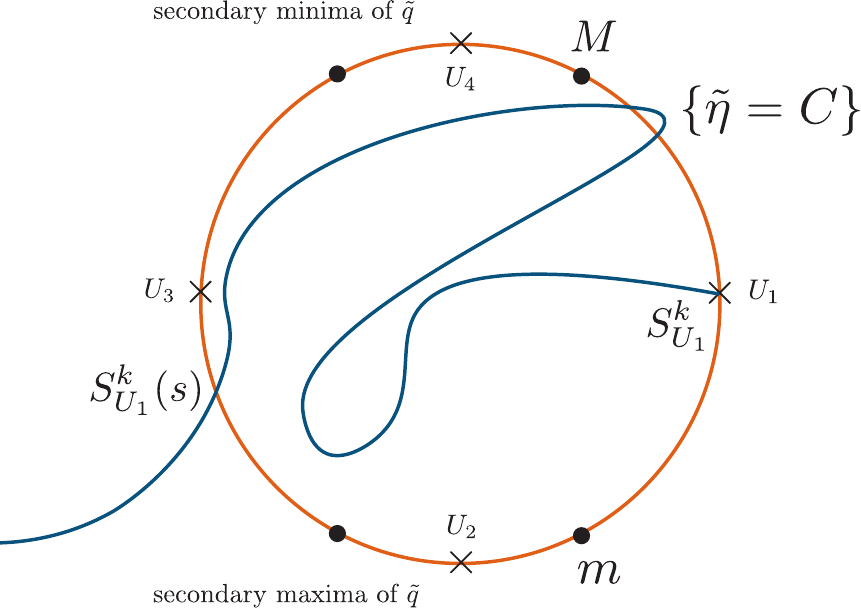}
  \caption{An illustration from the proof of \Cref{main_theorem}.}\label{four_crit_proof_fig}
\end{figure}

Thus, $S^k_{U_1}(s)$ must be somewhere between the four $U_i$, in particular between the same two of the $U_i$ as one of the two local maxima. 

We now argue that $S^k_{U_1}(s)$ is between the same  two $U_i$ as the global maxima. Assume this is not the case. Then, imagine dragging the starting point for the curve $S^k_{U_1}(s)$ along the level set $\{\tilde\eta = C\}$, moving the initial point from $U_1$ to the global maximum $M$. Then, by \Cref{curve_perturb}, we know that the $k$-shock curve emanating from the global maxima $M$ must make its last exit from the set $\{\tilde\eta\leq C\}$ between the  two of the $U_i$ which have in between them the maxima which is not $M$ (the second maxima). For an illustration, see \Cref{four_crit_proof_fig}. This is because, by \Cref{curve_perturb}, \Cref{diss_lemma} and \eqref{ordering_qEq}, as we perturb the initial point  $U_1$, the place where the curve $S^k$ last exits $\{\tilde\eta\leq C\}$ must stay between the same two of the $U_i$ (these $U_i$ ``lock'' the shock curve in place, preventing the shock curve from perturbing into a situation where it never enters the level set $\{\tilde\eta\leq C\}$ at all).

This however gives a contradiction to \eqref{diss_form} because $M$ is a global maxima, so $\tilde q$ cannot take a higher value on the level set  $\{\tilde\eta = C\}$ between the  two of the $U_i$ which contain the maxima which is not $M$ (the second maxima). Thus, we conclude $S^k_{U_1}(s)$ is between the same  two $U_i$ as the global maxima.

Similarly, in the case $\tilde q (S^k_{U_1}(s)) < \tilde q(U_1)$, we imagine dragging the starting point for the curve $S^k_{U_1}(s)$ along the level set $\{\tilde\eta = C\}$, moving the initial point from $U_1$ to the global minimum $m$.  

Hence, we can conclude that any shock curves emanating from $U_1$ (and entering $\{\tilde\eta\leq C\}$)  must make their last exits from the set $\{\tilde\eta\leq C\}$ between  $U_1$ and $m$ or between $U_1$ and $M$. 

From \Cref{TN_change_sign}, we have reached a contradiction and conclude that no such $T_4$ can exist.

\begin{figure}[tb]
      \includegraphics[width=.84\textwidth]{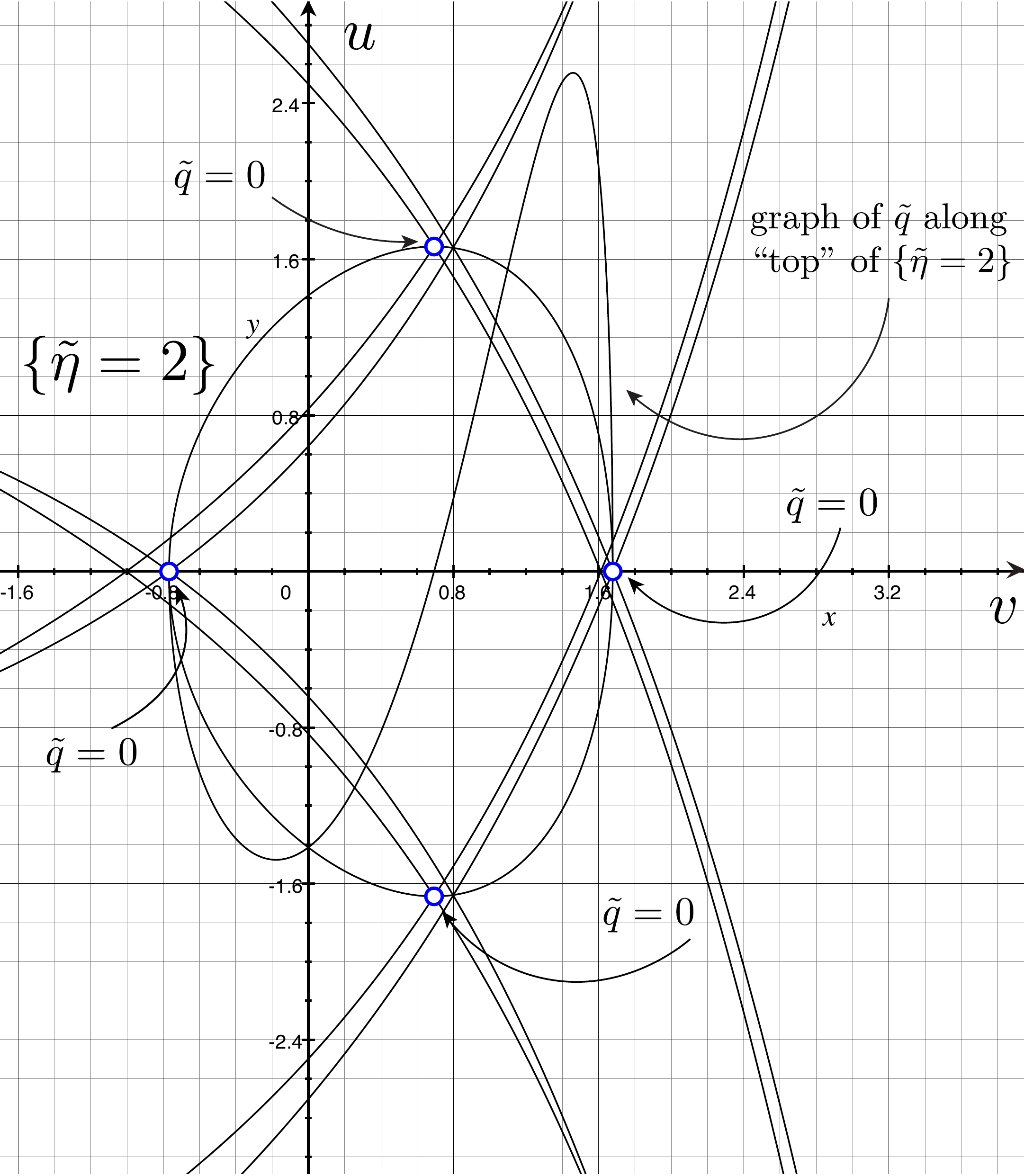}
  \caption{Consider the $p$-system \eqref{psystem}, for $p(v)=-e^v$, with entropy $\tilde\eta(v,u)=\frac{u^2}{2}+e^v-2u$ and  entropy-flux $\tilde q(v,u)=u(e^v-2)$. The figure shows the level set $\{\tilde\eta=2\}$. We consider the points on the level set $\{\tilde\eta=2\}$ where $\tilde q=0$. We have also included in this figure the graph of the function $\tilde q$ along the ``top'' of the level set $\{\tilde\eta=2\}$, which is parameterized by $u=\sqrt{4-2e^v+4v}$. Note that the graph of the function $\tilde q$ along the ``bottom'' of the level set is obtained by simply flipping this graph over the $v$-axis. Further, at each point where $\tilde q=0$, the union of the $S^1$ and $S^2$ curves is given.}\label{t4fig}
\end{figure}

The proof is illustrated with the example of the $p$-system \eqref{psystem} in \Cref{t4fig}. Remark that in this particular example, $\tilde q$ restricted to the level set of $\tilde\eta$ has exactly four extremal points.

\subsection{Proof of \Cref{changeofCoordsnoT4}}\label{sec:proofchangeofCoordsnoT4}

We argue by contradiction. 

Let the four points of a $T_4$ be $Z_1,\ldots, Z_4\in \mathcal{K}_{f,\eta,q}$.

By \Cref{lin_combo}, we know that there exists constants $c_1,c_2\in\mathbb{R}$ such that the system \eqref{system}, with the entropy $\tilde \eta$ and entropy-flux $\tilde q$ (see \eqref{tdef}), admits a $T_4$ configuration $X_1,\ldots, X_4\in \mathcal{K}_{f,\tilde\eta,\tilde q}$ with the property that $(X_i)_{3,1}=(X_j)_{3,1}$ and $(X_i)_{3,2}=(X_j)_{3,2}$ for all $j,k$. 

By changing the constant matrix $P$ in the context of \Cref{T_N_def} and adding a constant term to $\tilde\eta$, we can assume that at each of the $T_4$ points in $U$-space, $\tilde\eta=0$. 

Then, under the change of variables from \Cref{LEcoordinates}, the $T_4$ points will have corresponding $V$-plane coordinates in the new coordinate system. Let 
\begin{align}
f^{\mathrm{transformed}}=(f^{\mathrm{transformed}}_1,f^{\mathrm{transformed}}_2)
\end{align}
denote the flux in the new coordinates.

From \Cref{LEcoordinates}, we know that a shock $(U_L,U_R)$ verifying the Rankine-Hugoniot conditions for the system \eqref{system}, will be transformed to a shock $(V_L,V_R)$ verifying the Rankine-Hugoniot conditions for the transformed system. 

Remark that in the proofs of \Cref{main_theorem}, \Cref{main_theorem1}, the arguments procede by reducing to the case where the $T_4$ points are all on the same level sets of the entropy, and entropy-flux. This alone is enough to get a contradiction to the existence of a $T_4$ via the repeated use of \Cref{TN_change_sign}.

Finally, we note that due to the definition of $\hat\eta$ and $\hat q$ in \Cref{LEcoordinates}, and that $\tilde\eta=0$ on the $T_4$ points, the transformed $T_4$ points in $V$-space will also lie on the same level sets of the entropy $\hat\eta$ and entropy-flux $\hat q$. Let $V_1,\ldots, V_4$ denote the $V$-space coordinates of the four $T_4$ points after being transformed. Write $V_i=(v_i,w_i)$.

Then, define 
\begin{align}
Y_i\coloneqq
\begin{bmatrix}
v_i & f^{\mathrm{transformed}}_1(V_i) \\
w_i & f^{\mathrm{transformed}}_2(V_i)
\end{bmatrix}.
\end{align}

From the proofs of \Cref{main_theorem}, \Cref{main_theorem1}, we conclude that there exists $j$ such that $\det(Y_i-Y_j)$ has the same sign for all $i\neq j$. By continuity, we conclude that $\det(X_i-X_j)$ has the same sign for all $i\neq j$ (where we have restricted the determinant to the first two rows of the matrices $X_i$). This concludes the proof.

\section{Appendix}

\subsection{Proof of Lax's dissipation formula (\Cref{diss_lemma})}\label{proof_app}
We follow the proof in \cite{MR3537479}.
Write $u_R=S^k_{u_L}(s)$ for some $k=1,2$ and $s\in\mathbb{R}$. 

Define
\begin{align}
    \mathcal{F}_1(s)\coloneqq& q(S^k_{u_L}(s))-F(u_L),\\
    \mathcal{F}_2(s)\coloneqq& \sigma^k_{u_L}(s)\big(\eta(S^k_{u_L}(s))-\eta(u_L)\big)+\int\limits_0^s \frac{\text{d}}{\text{ds}}\big[\sigma^k_{u_L}\big](\tau)\eta(u_L|S^k_{u_L}(\tau))\,d\tau.
\end{align}

The proof is complete if we show $\mathcal{F}_1(s)=\mathcal{F}_2(s)$ for all $s$. Due to $S^k_{u_L}(0)=u_L$, the equality is true for $s=0$.

We then compute,
\begin{align}
    \mathcal{F}_1'(s)&=\nabla\eta(S^k_{u_L}(s))\frac{\text{d}}{\text{ds}}(f(S^k_{u_L}(s))-f(u_L)),
\shortintertext{and}
    \mathcal{F}_2'(s)&= \frac{\text{d}}{\text{ds}}\big[\sigma^k_{u_L}\big](s)(-\nabla\eta(S^k_{u_L}(s))(u_L-S^k_{u_L}(s))) 
     \\&\hspace{.3in}+ \sigma^k_{u_L}(s)(\nabla\eta(S^k_{u_L}(s))\frac{\text{d}}{\text{ds}}S^k_{u_L}(s))\\
    &=\nabla\eta(S^k_{u_L}(s))\frac{\text{d}}{\text{ds}}(\sigma^k_{u_L}(s)(S^k_{u_L}(s)-u_L)).
\end{align}

Then, we recall that by definition we have the Rankine-Hugoniot condition for the shock $(u_L,S^k_{u_L}(s),\sigma^k_{u_L}(s))$,
\begin{align}
    \sigma^k_{u_L}(s)(S^k_{u_L}(s)-u_L)=f(S^k_{u_L}(s))-f(u_L).
\end{align}
Thus, $\mathcal{F}_1'(s)=\mathcal{F}_2'(s)$ for $s\neq 0$ and this completes the proof.

\subsection{Equivalence of weak solutions for the  Eulerian and Langrangian equations}\label{equivappendix}
For the convenience of the reader, we reproduce nearly verbatim the main theorem on the equivalence of weak solutions for the  Eulerian and Langrangian equations of gas dynamics from \cite[Theorem 2]{wagnergasdynamics}.

\begin{theorem}[\protect{Equivalence of Eulerian and Lagrangian coordinates \cite[Theorem 2]{wagnergasdynamics}}]\label{LEcoordinates}
\hspace{.1in}

Let
\begin{align}\label{LEsystem}
    \partial_t U +\partial_x F(U)=0
\end{align}
for $U(x,t)=(u_1,\ldots,u_n)$, $(x,t)\in\mathbb{R}\times[0,\infty)$, $F(U)=(f_1,\ldots,f_n)$, $U\in\mathbb{R}^n$, be a system of conservation laws. For any bounded measurable solution of \eqref{LEsystem}, with $u_1(x,t)\geq 0$, let $y(x,t)$ satisfy
\begin{align}
    \frac{\partial y}{\partial x} =u_1(x,t), \hspace{.3in} \frac{\partial y}{\partial t}=-f_1(U(x,t)),
\end{align}
in the sense of distributions. Then $T\colon (x,t)\mapsto (y(x,t),t)$ is a Lipschitz-continuous transformation, which induces a one-to-one correspondence between $L^\infty$ weak solutions of \eqref{LEsystem} on $\mathbb{R}\times[0,\infty)$ satisfying $0<\epsilon\leq u_1(x,t)\leq M<\infty$ for some $\epsilon$ and $M$, and $L^\infty$ weak solutions of
\begin{align}\label{lagrangian_coords}
     \begin{cases}
  \partial_t (1/u_1) - \partial_y (f_1(U)/u_1)=0,\\
  \partial_t \Big[(u_2,\ldots,u_n)/u_1\Big]+\partial_y\Big[(f_2,\ldots,f_n)(U)-f_1(U)(u_2,\ldots,u_n)/u_1\Big]=0
    \end{cases}
\end{align}
on $\mathbb{R}\times[0,\infty)$ satisfying $\epsilon\leq u_1(x,t)\leq M$. In addition, if $F(U)/u_1$ is bounded for $u_1>0$, then there is a one-to-one correspondence between equivalence classes of bounded measurable solutions of \eqref{LEsystem} for which $u_j/u_1$ is bounded for $j=2,\ldots,n$ and
\begin{align}
    \int\limits_0^\infty u_1(x,t)\,dx=\int\limits_{-\infty}^0 u_1(x,t)\,dx=\infty,
\end{align}
and equivalence classes of weak solutions of \eqref{lagrangian_coords} for which $v_1=1/u_1$ is a Radon measure which dominates Lebesgue measure\footnote{See \cite{wagnergasdynamics} for details.}, and $v_j=u_j/u_1$ is bounded for $j=2,\ldots,n$. If $\eta(U)$ is any convex extension of \eqref{LEsystem}, i.e., there is an entropy-flux $q(U)$ such that $\nabla\eta DF=\nabla q$, so that $\partial_t \eta+\partial_x q=0$ for classical solutions, then any solution of \eqref{LEsystem} satisfying 
\begin{align}
    \partial_t \eta(U)+\partial_x q(U)\leq 0
\end{align}
corresponds to a solution of \eqref{lagrangian_coords} satisfying 
\begin{align}
    \partial_t \hat\eta(U)+\partial_x \hat q(U)\leq 0,
\end{align}
where $V=(v_1,\ldots,v_n)$, $\hat\eta(V)=\eta(U)/u_1$, and $\hat q(V)=q(U)-f_1(U)\hat\eta(V)$. Furthermore $\eta$ is convex if and only if $\hat\eta$ is convex as a function of $V$. 
\end{theorem}

\bibliographystyle{plain}
\bibliography{references}
\end{document}